\documentclass[11pt,reqno,tbtags]{amsart}
\usepackage{amssymb}
\usepackage[utf8]{inputenc}   
\usepackage{url}
\usepackage[square,numbers]{natbib}

\numberwithin{equation}{section}

\renewcommand\le{\leqslant}
\renewcommand\ge{\geqslant}

\allowdisplaybreaks

\newtheorem{theorem}{Theorem}[section]

\newtheorem{corollary}[theorem]{Corollary}

\theoremstyle{definition}

\newtheorem{exampleqqq}[theorem]{Example}
\newenvironment{example}{\begin{exampleqqq}}
  {\hfill\qedsymbol\end{exampleqqq}}

\newtheorem{remarkqqq}[theorem]{Remark}
\newenvironment{remark}{\begin{remarkqqq}}
  {\hfill\qedsymbol\end{remarkqqq}}

\newtheorem{definition}[theorem]{Definition}

\theoremstyle{remark}

\newenvironment{romenumerate}[1][0pt]{
\addtolength{\leftmargini}{#1}\begin{enumerate}
 }{\end{enumerate}}

\newcounter{oldenumi}
{\setcounter{oldenumi}{\value{enumi}}
\begin{romenumerate} \setcounter{enumi}{\value{oldenumi}}}
{\end{romenumerate}}

\newcounter{thmenumerate}

\newcounter{romxenumerate}   

\newcounter{xenumerate}   

\newcommand\pfitemx[1]{\par#1:}
\newcommand\pfitemref[1]{\pfitemx{\ref{#1}}}

\newcommand{\refT}[1]{Theorem~\ref{#1}}
\newcommand{\refTs}[1]{Theorems~\ref{#1}}
\newcommand{\refC}[1]{Corollary~\ref{#1}}

\newcommand{\refR}[1]{Remark~\ref{#1}}
\newcommand{\refS}[1]{Section~\ref{#1}}

\newcommand{\refD}[1]{Definition~\ref{#1}}
\newcommand{\refE}[1]{Example~\ref{#1}}
\newcommand{\refEs}[1]{Examples~\ref{#1}}

\begingroup
  \divide\count255 by 60
  \multiply\count255 by -60
  \advance\count255 by \time
  \ifnum \count255 < 10 \xdef\klockan{\the\count1.0\the\count255}
  \else\xdef\klockan{\the\count1.\the\count255}\fi
\endgroup

\newcommand\nopf{\qed}   

\newcommand{\sumki}{\sum_{k=1}^\infty}

\newcommand{\sumn}{\sum_{n=0}^\infty}

\newcommand{\sumin}{\sum_{i=1}^n}

\newcommand\set[1]{\ensuremath{\{#1\}}}

\newcommand\xpar[1]{(#1)}
\newcommand\bigpar[1]{\bigl(#1\bigr)}
\newcommand\Bigpar[1]{\Bigl(#1\Bigr)}
\newcommand\biggpar[1]{\biggl(#1\biggr)}
\newcommand\lrpar[1]{\left(#1\right)}
\newcommand\sqpar[1]{[#1]}
\newcommand\bigsqpar[1]{\bigl[#1\bigr]}

\newcommand\lrsqpar[1]{\left[#1\right]}
\newcommand\xcpar[1]{\{#1\}}

\newcommand\bigabs[1]{\bigl|#1\bigr|}

\def\rompar(#1){\textup(#1\textup)}    
\newcommand\xfrac[2]{#1/#2}

\newcommand\parfrac[2]{\lrpar{\frac{#1}{#2}}}

\newcommand\Bigparfrac[2]{\Bigpar{\frac{#1}{#2}}}

\def\xexp(#1){e^{#1}}

\newcommand\floor[1]{\lfloor#1\rfloor}

\newcommand\nn{^{(n)}}
\newcommand\ntoo{\ensuremath{{n\to\infty}}}

\newcommand\xtoo{\ensuremath{{x\to\infty}}}
\newcommand\ttoo{\ensuremath{{t\to\infty}}}
\newcommand\bmin{\wedge}

\newcommand\downto{\searrow}
\newcommand\upto{\nearrow}

\newcommand\punkt[1]{\if.#1\else.\spacefactor1000\fi{#1}}
\newcommand\iid{i.i.d\punkt}    
\newcommand\ie{i.e\punkt}
\newcommand\eg{e.g\punkt}

\newcommand{\as}{a.s\punkt}

\newcommand\ii{\mathrm{i}}

\newcommand{\tend}{\longrightarrow}
\newcommand\dto{\overset{\mathrm{d}}{\tend}}
\newcommand\pto{\overset{\mathrm{p}}{\tend}}

\newcommand\eqd{\overset{\mathrm{d}}{=}}

\newcommand\bbR{\mathbb R}
\newcommand\bbC{\mathbb C}

\newcounter{CC}
\newcounter{cc}

\renewcommand\Re{\operatorname{Re}}
\renewcommand\Im{\operatorname{Im}}

\newcommand\E{\operatorname{\mathbb E{}}}
\renewcommand\P{\operatorname{\mathbb P{}}}
\newcommand\Var{\operatorname{Var}}

\newcommand\Exp{\operatorname{Exp}}
\newcommand\Po{\operatorname{Po}}

\newcommand\supp{\operatorname{supp}}
\newcommand\sgn{\operatorname{sgn}}

\newcommand\ga{\alpha}
\newcommand\gb{\beta}
\newcommand\gd{\delta}

\newcommand\gf{\varphi}
\newcommand\gam{\gamma}
\newcommand\gG{\Gamma}

\newcommand\gl{\lambda}
\newcommand\gL{\Lambda}
\newcommand\go{\omega}

\newcommand\gs{\sigma}
\newcommand\gss{\sigma^2}
\newcommand\gth{\theta}
\newcommand\eps{\varepsilon}

\renewcommand\phi{\xxx}  

\newcommand\cA{\mathcal A}

\newcommand\cE{\mathcal E}

\newcommand\cL{{\mathcal L}}

\newcommand\ett[1]{\boldsymbol1\xcpar{#1}}

\newcommand\qw{^{-1}}
\newcommand\qww{^{-2}}
\newcommand\qq{^{1/2}}

\renewcommand{\=}{:=}

\newcommand\intoi{\int_0^1}
\newcommand\intoo{\int_0^\infty}
\newcommand\intoooo{\int_{-\infty}^\infty}
\newcommand\oi{[0,1]}
\newcommand\ooo{[0,\infty)}
\newcommand\oooy{(0,\infty)}

\newcommand\dd{\,\mathrm{d}}

\newcommand{\mgf}{moment generating function}
\newcommand{\chf}{characteristic function}

\newcommand\rv{random variable}
\newcommand\lhs{left-hand side}
\newcommand\rhs{right-hand side}

\newcommand\etto{\bigpar{1+o(1)}}
\newcommand\id{infinitely divisible}
\newcommand\idd{infinitely divisible distribution}

\newcommand\xga{^{1/\ga}}
\newcommand\xgaw{^{-1/\ga}}
\newcommand\Lm{\Levy{} measure}
\newcommand\gssa{a}
\newcommand\hx{x\ett{|x|\le 1}}
\newcommand\ID{\mathrm{ID}}
\newcommand\piaq{\frac{\pi\ga}{2}}
\newcommand\tpiaq{\tfrac{\pi\ga}{2}}
\newcommand\tgam{\widetilde\gam}
\newcommand\pigq{\frac{\pi\tgam}{2}}
\newcommand\gamx{\gam^*}
\newcommand\tX{\widetilde X}
\newcommand\ggam{\bar\gamma}
\newcommand\kk{\kappa}
\newcommand\Sx{\mathrm{S}}
\newcommand\SSSS[4]{\Sx_{#1}(#2,#3,#4)}
\newcommand\SSSx[1]{\Sx_{#1}}
\newcommand\XXXX[4]{X_{#1}(#2,#3,#4)}
\newcommand\XXXx[1]{X_{#1}}
\newcommand\zz{\bullet}
\newcommand\U{\mathrm{U}}
\newcommand\glx{\gl}
\newcommand\gld{\gl}
\newcommand\Ai{\mathrm{Ai}}
\newcommand\YY{\mathrm Y}

\newcommand{\Levy}{L\'evy}

\newcommand\urladdrx[1]{{\urladdr{\def~{{\tiny$\sim$}}#1}}}

\begin{document}
\title
{Stable distributions}

\date{1 December, 2011; revised 24 February, 2022}


\author{Svante Janson}
\address{Department of Mathematics, Uppsala University, PO Box 480,
SE-751~06 Uppsala, Sweden}
\email{svante.janson@math.uu.se}
\urladdrx{http://www.math.uu.se/~svante/}



\maketitle

\section{Introduction}\label{S:intro}
We give many explicit formulas for stable distributions, 
mainly based on \citet{FellerII} and \citet{Z} and
using several parametrizations;
we give also some explicit  calculations for convergence
to stable distributions,  mainly based on less explicit results
in \citet{FellerII}. 
The main purpose is to provide ourselves with easy reference to explicit
formulas and examples.
(There are probably no new results.)

\section{Infinitely divisible distributions}\label{Sid}

We begin with the more general concept of \id{} distributions.

\begin{definition}
The distribution of a 
random variable $X$ is \emph{\id} if for each $n\ge1$ there exists
\iid{} \rv{} $Y\nn_1,\dots,Y\nn_n$ such that
\begin{equation}\label{id}
 X\eqd Y\nn_1+\dots Y\nn_n.
\end{equation}
\end{definition}

The \chf{} of an \idd{} may be expressed in a canonical form, sometimes
called the 
\emph{\Levy--Khinchin representation}.
We give several equivalent versions in the following theorem.

\begin{theorem}\label{Tid}
Let $h(x)$ be a fixed bounded measurable
real-valued function on $\bbR$ such that
$h(x)=x+O(x^2)$ as $x\to0$.
Then the following are equivalent.
\begin{romenumerate}[-10pt]
\item \label{tidchf}
$\gf(t)$ is the \chf{} of an \idd.
\item \label{tidcanonical}
There exist a measure $M$ on $\bbR$ such that
\begin{equation}\label{tcan0}
  \intoooo \bigpar{1\bmin |x|\qww}\dd M(x)<\infty
\end{equation}
and a real constant $b$ such that
\begin{equation}\label{tcan}
  \gf(t)
=
\exp\Bigpar{\ii bt+\intoooo \frac{e^{\ii tx}-1-\ii th(x)}{x^2}\dd M(x)},
\end{equation}
where the integrand is interpreted as $-t^2/2$ at $x=0$.

\item \label{tidlevy}
There exist a measure $\gL$ on $\bbR\setminus\set0$ such that
\begin{equation}\label{tlevy0}
  \intoooo \bigpar{|x|^{2}\bmin 1}\dd \gL(x)<\infty
\end{equation}
and real constants $\gssa\ge0$ and $b$ such that
\begin{equation}\label{tlevy}
  \gf(t)
=
\exp\Bigpar{\ii bt-\tfrac12\gssa t^2+\intoooo \bigpar{e^{\ii tx}-1-\ii th(x)}\dd
  \gL(x)}.
\end{equation}

\item \label{tidkh}
There exist a bounded measure $K$ on $\bbR$ 
and a real constant $b$ such that
\begin{equation}\label{tkh}
  \gf(t)
=
\exp\Bigpar{\ii bt+\intoooo 
\Bigpar{e^{\ii tx}-1-\frac{\ii tx}{1+x^2}}\frac{1+x^2}{x^2}\dd K(x)},
\end{equation}
where the integrand is interpreted as $-t^2/2$ at $x=0$.
\end{romenumerate}
The measures and constants are determined uniquely by $\gf$.
\end{theorem}

\citet[Chapter XVII]{FellerII} 
uses $h(x)=\sin x$.
\citet[Corollary 15.8]{Kallenberg} uses $h(x)=\hx$.

\citet[Chapter XVII.2]{FellerII} calls the measure $M$ in
\ref{tidcanonical}
the \emph{canonical measure}.
The measure $\gL$ in \ref{tidlevy} is known as the \emph{\Levy{} measure}.
The parameters
$a$, $b$ and $\gL$ are together
called the \emph{characteristics} of the distribution.
We denote the distribution with \chf{} \eqref{tlevy} (for a given $h$) by
$\ID(a,b,\gL)$.

\begin{remark}
Different choices of $h(x)$ 
yield the same measures $M$ and $\gL$
in \ref{tidcanonical} and \ref{tidlevy} 
but different constants $b$; changing $h$ to $\tilde h$
corresponds to changing $b$ to
\begin{equation}\label{tb}
  \tilde b 
\= b +\intoooo\frac{\tilde h(x)-h(x)} {x^2}\dd M(x)
=  b +\intoooo\bigpar{\tilde h(x)-h(x)}\dd \gL(x).
\end{equation}
We see also that $b$ is the same
in \ref{tidcanonical} and \ref{tidlevy} (with the same $h$),
and that (see the proof below) $b$ in \ref{tidkh} equals $b$ 
in \ref{tidcanonical} and \ref{tidlevy} when $x=x/(1+x^2)$. 
\end{remark}

\begin{proof}
  \ref{tidchf}$\iff$\ref{tidcanonical}:
This is shown in \citet[Theorem XVII.2.1]{FellerII} for the choice
$h(x)=\sin x$. As remarked above, \eqref{tcan} for some $h$ is equivalent to
\eqref{tcan} for any other $h$, changing $b$ by \eqref{tb}.

\ref{tidcanonical}$\iff$\ref{tidlevy}:
Given $M$ in \ref{tidcanonical}
we let $\gssa\=M\set0$ and $\dd\gL(x)\=x^{-2}\dd M(x)$, $x\neq0$.
Conversely, given $\gssa$ and $\gL$ as in \ref{tidlevy} we define
\begin{equation}
  \dd M(x) = \gssa\gd_0 + x^2\dd\gL(x).
\end{equation}
The equivalence between \eqref{tcan} and \eqref{tlevy} then is obvious.
\end{proof}

\ref{tidcanonical}$\iff$\ref{tidkh}:
Choose $h(x)=x/(1+x^2)$ and define
\begin{equation}
  \dd K(x) \= \frac{1}{1+x^2} \dd M(x);
\end{equation}
conversely, $\dd M(x)=(1+x^2)\dd K(x)$.
Then \eqref{tcan} is equivalent to \eqref{tkh}.

\begin{remark}\label{Rproc}
  At least \ref{tidlevy} extends directly to \id{} random vectors in
  $\bbR^d$.
Moreover, 
  there is a one-to-one correspondence with \emph{\Levy{} processes}, \ie,
  stochastic processes $X_t$ on $\ooo$ with stationary independent
  increments and $X_0=0$,  given by (in the one-dimensional case)
\begin{equation}\label{tlevyp}
  \E e^{\ii u X_t}
=\gf(u)^t=
\exp\Bigpar{t\Bigpar{\ii bu-\tfrac12\gssa u^2
 +\intoooo \bigpar{e^{\ii ux}-1-\ii uh(x)}\dd  \gL(x)}}
\end{equation}
for $t\ge0$ and $u\in\bbR$. 
See \citet{Bertoin} and 
\citet[Corollary 15.8]{Kallenberg}. 
\end{remark}

\begin{example}\label{Enormal}
  The normal distribution $N(\mu,\gss)$ has 
$\gL=0$ and $\gssa=\gss$; thus
$M =K= \gss\gd_0$; further, $b=\mu$ for any $h$.  
Thus, 
$N(\mu,\gss)=\ID(\gss,\mu,0)$.
\end{example}

\begin{example}
  The Poisson distribution $\Po(\gl)$ has $M=\gL=\gl\gd_1$ and
  $K=\frac{\gl}2\gd_1$; further $b=\gl h(1)$. (Thus $b=\gl/2$ in \ref{tidkh}.)
\end{example}

\begin{example}
  The Gamma distribution Gamma$(\ga)$ with density function
  $x^{\ga-1}e^{-x}/\gG(\ga)$, $x>0$, has the \chf{} 
$\gf(t)=(1-\ii  t)^{-\ga}$.
It is \id{} with 
\begin{align}
 \dd M(x)& = \ga xe^{-x},\qquad  x>0,
\\
 \dd \gL(x)& = \ga x\qw e^{-x},\qquad  x>0,
\end{align}
see \citet[Example XVII.3.d]{FellerII}.
\end{example}

\begin{remark}\label{Radd}
  If $X_1$ and $X_2$ are independent \id{} \rv{s} with parameters
  $(a_1,b_1,\gL_1)$  and   $(a_2,b_2,\gL_2)$, then
$X_1+X_2$ is \id{} with parameters
  $(a_1+a_2,b_1+b_2,\gL_1+\gL_2)$.
In particular, if $X\sim\ID(a,b,\gL)$, then 
\begin{equation}
  X\eqd X_1+Y
\qquad\text{with}\qquad
X_1 \sim \ID(0,0,\gL), 
\;
Y\sim \ID(a,b,0)=N(b,a),
\end{equation}
and $X_1$ and $Y $ independent.
Moreover, for any finite partition $\bbR=\bigcup A_i$,
we can split $X$ as a sum of independent \id{} \rv{s} $X_i$ with the \Lm{}
of $X_i$ having supports in $A_i$.
\end{remark}

\begin{example}[integral of Poisson process]\label{Epop}
  Let $\Xi$ be a Poisson process on $\bbR\setminus\set0$ with intensity
  $\gL$, where $\gL$ is a measure with 
\begin{equation}\label{tlevy1}
  \intoooo \bigpar{|x|\bmin1}\dd \gL(x)<\infty.
\end{equation}
Let $X\=\int x\dd\Xi(x)$; if we regard $\Xi$ as a (finite or countable) set
(or possibly multiset) of points \set{\xi_i}, this means that
$X\=\sum_i\xi_i$.
(The sum converges absolutely \as, so $X$ is well-defined \as; in fact, the
sum $\sum_{|\xi_i|>1}\xi_i$ is \as{} finite, and the sum 
$\sum_{|\xi_i|\le1}|\xi_i|$ has finite expectation 
$\int_{-1}^1 |x|\dd\gL(x)$.)
Then $X$ has \chf{}
\begin{equation}
  \label{po0}
\gf(t)=\exp\Bigpar{\intoooo\bigpar{e^{\ii tx}-1}\dd\gL(x)}.
\end{equation}
(See, for example, the corresponding formula for the Laplace transform in
\citet[Lemma 12.2]{Kallenberg}, from which \eqref{po0} easily follows.)
Hence, \eqref{tlevy} holds with \Lm{} $\gL$, $\gssa=0$ and $b=\intoooo
h(x)\dd\gL(x)$. 
(When \eqref{tlevy} holds, we can take $h(x)=0$, a choice not allowed in
general. Note that \eqref{po0} is the same as \eqref{tlevy} with $h=0$,
$a=0$ and $b=0$.)

By adding an independent normal variable $N(b,a)$, we can obtain any \id{}
distribution with a \Lm{} satisfying \eqref{tlevy}; see 
\refE{Enormal} and
\refR{Radd}.
\end{example}

\begin{example}[compensated integral of Poisson process]\label{Ecpo}
  Let $\Xi$ be a Poisson process on $\bbR\setminus\set0$ with intensity
  $\gL$, where $\gL$ is a measure with 
\begin{equation}\label{tlevy2}
  \intoooo \bigpar{|x|^2\bmin|x|}\dd \gL(x)<\infty.
\end{equation}
Suppose first that $\intoooo |x|\dd \gL(x)<\infty$. Let $X$ be as in
\refE{Epop}. Then $X$ has finite expectation $\E X=\intoooo x\dd\gL$.
Define 
\begin{equation}
  \label{cpo}
\tX\=X-\E X =\intoooo x\bigpar{\dd\Xi(x)-\dd\gL(x)}.
\end{equation}
Then, by \eqref{po0},
 $\tX$ has \chf{}
\begin{equation}
  \label{po1}
\gf(t)=\exp\Bigpar{\intoooo\bigpar{e^{\ii tx}-1-itx}\dd\gL(x)}.
\end{equation}

Now suppose that $\gL$ is any measure satisfying \eqref{tlevy2}.
Then the integral in \eqref{po1} converges; moreover, by considering the
truncated measures $\gL_n\=\ett{|x|>n\qw}\gL$ and taking the limit as \ntoo,
it follows that there exists a \rv{} $\tX$ with \chf{} \eqref{po1}.
Hence, \eqref{tlevy} holds with \Lm{} $\gL$, $\gssa=0$ and $b=\intoooo
(h(x)-x)\dd\gL(x)$. 
(When \eqref{tlevy2} holds, we can take $h(x)=x$, a choice not allowed in
general. Note that \eqref{po1} is the same as \eqref{tlevy} with $h(x)=x$,
$a=0$ and $b=0$.)

By adding an independent normal variable $N(b,a)$, we can obtain any \id{}
distribution with a \Lm{} satisfying \eqref{tlevy2}; see 
\refE{Enormal} and
\refR{Radd}.
\end{example}

\begin{remark}
Any \idd{} can be obtained by taking a sum $X_1+X_2+Y$ of
independent \rv{s} with 
$X_1$ as in \refE{Epop},
$X_2$ as in \refE{Ecpo}
and $Y$ normal.
For example, we can take the \Lm{s} of $X_1$ and $X_2$ as the restrictions
of the \Lm{} to $\set{x:|x|>1}$ and $\set{x:|x|\le1}$, respectively.  
\end{remark}

\begin{theorem}\label{Tidexp}
  If $X$ is an \id{} \rv{} with \chf{} given by \eqref{tlevy} and
  $t\in\bbR$, then
\begin{equation}\label{exp}
\E e^{tX}
=
\exp\Bigpar{bt+\tfrac12\gssa t^2+\intoooo \bigpar{e^{tx}-1-th(x)}\dd \gL(x)}
\le\infty.
\end{equation}
In particular, 
\begin{equation}
  \begin{split}
\E e^{tX}<\infty
&\iff
\intoooo \bigpar{e^{tx}-1-th(x)}\dd \gL(x)<\infty
\\&
\iff
\begin{cases}
  \int_1^\infty e^{tx}\dd\gL(x)<\infty, & t>0,\\
  \int_{-\infty}^{-1} e^{tx}\dd\gL(x)<\infty, & t<0.
\end{cases}	
  \end{split}
\end{equation}
\end{theorem}
\begin{proof}
  The choice of $h$ (satisfying the conditions of \refT{Tid}) does not
  matter, because of \eqref{tb}; we may thus assume $h(x)=\hx$.
We further assume  $t>0$. (The case $t<0$ is similar and the case $t=0$ is
trivial.) 

Denote the \rhs{} of \eqref{exp} by $F_\gL(t)$. We study several different
cases.

(i). If $\supp\gL$ is bounded, then the integral in \eqref{exp} converges for
all complex $t$ and defines an entire function. Thus $F_\gL(t)$ is entire
and \eqref{tlevy} shows that $\E e^{\ii tX}=F_\gL(\ii t)$. It follows that
$\E|e^{tX|}<\infty$ and $\E e^{tX}=F_\gL(t)$ for any complex $t$, see \eg{}
\citet{Marcinkiewicz}. 

(ii). If $\supp\gL\subseteq[1,\infty)$, let $\gL_n$ be the restriction
  $\gL\big|_{[1,n]}$ of the measure $\gL$ to $[1,n]$.
By the construction in \refE{Epop}, we can construct random variables
$X_n\sim\ID(0,0,\gL_n)$ such that $X_n\upto X\sim\ID(0,0,\gL)$ as \ntoo.
Case (i) applies to each $\gL_n$, and \eqref{exp} follows for $X$, and
$t>0$, by monotone convergence.

(iii). If $\supp\gL\subseteq(-\infty,1]$, let $\gL_n$ be the restriction
  $\gL\big|_{[-n,-1]}$. Similarly to (ii)
we can construct random variables
$X_n\sim\ID(0,0,\gL_n)$ with $X_n\le0$ such that 
$X_n\downto X\sim\ID(0,0,\gL)$ as \ntoo.
Case (i) applies to each $\gL_n$, and \eqref{exp} follows for $X$; this time
by monotone convergence. 

(iv). The general case follows by (i)--(iii) and a decomposition as in
\refR{Radd}. 
\end{proof}

\section{Stable distributions}\label{Sstab}

\begin{definition}
The distribution of a (non-degenerate)
random variable $X$ is \emph{stable} if there exist constants
$a_n>0$ and $b_n$ such that, for any $n\ge1$, if $X_1,X_2,\dots$ are \iid{}
copies of $X$ and $S_n\=\sumin X_i$, then
\begin{equation}\label{stab}
  S_n\eqd a_n X+b_n.
\end{equation}
The distribution is \emph{strictly stable} if $b_n=0$.
\end{definition}

(Many authors, \eg{} \citet{Kallenberg}, 
say \emph{weakly stable} for our stable.)

We say that the random variable $X$ is (strictly) stable if its
distribution is.

The norming constants $a_n$ in \eqref{stab}
are necessarily of the form $a_n=n\xga$ for some
$\ga\in(0,2]$,
see \citet[Theorem VI.1.1]{FellerII}; 
$\ga$ is called the \emph{index} \cite{Gut}, \cite{Kallenberg}
or \emph{characteristic exponent} \cite{FellerII} of the distribution.
We also say that a distribution (or random variable) is \emph{$\ga$-stable}
if it is stable with index $\ga$.

The case $\ga=2$ is simple: $X$ is 2-stable if and only if it is normal.
For $\ga<2$, there is a simple characterisation in terms of the
\Levy--Khinchin representation of \idd{s}.

\begin{theorem}\label{Ts}
  \begin{romenumerate}[-10pt]
  \item 
A distribution is 2-stable if and only if it is normal $N(\mu,\gss)$.
(This is an \idd{} with $M=\gss\gd_0$, see \refE{Enormal}.)
\item 
Let $0<\ga<2$.
  A distribution is $\ga$-stable if and only if it is \id{} with 
canonical measure 
\begin{equation}\label{tsM}
\frac{ \dd M(x)}{\dd x} =
  \begin{cases}
	c_+ x^{1-\ga}, & x>0,
\\
	c_- |x|^{1-\ga}, & x<0;
  \end{cases}
\end{equation}
equivalently, the
\Levy{} measure is given by 
\begin{equation}\label{tsL}
 \frac{\dd\gL(x)}{\dd x} =
  \begin{cases}
	c_+ x^{-\ga-1}, & x>0,
\\
	c_- |x|^{-\ga-1}, & x<0,
  \end{cases}
\end{equation}
and $\gssa=0$. Here $c_-,c_+\ge0$ and we assume that not both are $0$.
  \end{romenumerate}
\end{theorem}
\begin{proof}
  See \citet[Section XVII.5]{FellerII} or 
\citet[Proposition	15.9]{Kallenberg}. 
\end{proof}

Note that \eqref{tsM} is equivalent to
\begin{equation}\label{mc}
  M[x_1,x_2]= {C_+}x_2^{2-\ga}+{C_-}|x_1|^{2-\ga}
\end{equation}
for any interval with $x_1\le0\le x_2$, with
\begin{equation}\label{erika}
C_\pm=\frac{c_\pm}{2-\ga}.
\end{equation}

\begin{theorem}\label{Tchf}
Let $0<\ga\le2$.
\begin{romenumerate}
\item 
A distribution is $\ga$-stable if and only if it 
has a \chf{}  
\begin{equation}\label{chf}
 \gf(t)=
 \begin{cases}
\exp\Bigpar{-\gam^\ga|t|^\ga
\Bigpar{1-\ii\gb\tan\frac{\pi\ga}2\sgn(t)}
+\ii\gd t}, &\ga\neq1,
\\
\exp\Bigpar{-\gam|t|
\Bigpar{1+\ii\gb\frac2{\pi}\sgn(t)\log|t|}
+\ii\gd t}
, & \ga=1,
 \end{cases}
\end{equation}  
where $-1\le\gb\le1$, $\gam>0$ and $-\infty<\gd<\infty$.
Furthermore, an $\ga$-stable distribution exists for any such
$\ga,\gb,\gam,\gd$. 
(If $\ga=2$, then $\gb$ is irrelevant and usually taken as $0$.)
\item 
If $X$ has the \chf{} \eqref{chf}, then, for any $n\ge1$,
\eqref{stab} takes the explicit form
\begin{equation}
 S_n\eqd
 \begin{cases}
n\xga X+(n-n\xga)\gd, &\ga\neq1,
\\
nX+\frac2\pi\gb\gam n\log n, & \ga=1.
 \end{cases}
\end{equation}
In particular,
\begin{equation}\label{tchfb}
X \text{ is strictly stable}
\iff
 \begin{cases}
\gd=0, &\ga\neq1,
\\
\gb=0, & \ga=1.
 \end{cases}
\end{equation}

\item 
An $\ga$-stable distribution with canonical measure $M$ satisfying
\eqref{mc} has
\begin{align} \label{chf2a}
  \gam^\ga&=
  \begin{cases}
(C_++C_-)\frac{\Gamma(3-\ga)}{\ga(1-\ga)}\cos\frac{\pi\ga}2, &\ga\neq1,\\
(C_++C_-)\frac{\pi}2, &\ga=1,
  \end{cases}
\\	
\gb&=\frac{C_+-C_-}{C_++C_-}. \label{chf2b}
\end{align}

\item If $0<\ga<2$, then 
an $\ga$-stable distribution with \Lm{} $\gL$ satisfying
\eqref{tsL} has
\begin{align}\label{chf3a}
  \gam^\ga&=
  \begin{cases}
(c_++c_-)\bigpar{-\Gamma(-\ga)\cos\frac{\pi\ga}2}, &\ga\neq1,\\
(c_++c_-)\frac{\pi}2, &\ga=1,
  \end{cases}
\\	
\gb&=\frac{c_+-c_-}{c_++c_-}. \label{chf3b}
\end{align}
\end{romenumerate}
\end{theorem}

We use the notation $\SSSS\ga\gam\gb\gd$ for the distribution with \chf{}
\eqref{chf},
and $\XXXX\ga\gam\gb\gd$ for a random variable with this distribution.
We also write $\SSSx\ga(\gb)$ and $\XXXx\ga(\gb)$ for the special case
$\gam=1$, $\gd=0$.

\begin{proof}
\citet[XVII.(3.18)--(3.19) and Theorem XVII.5.1(ii)]{FellerII}
gives, in our notation, for a stable distribution satisfying \eqref{mc},
the \chf{}
\begin{equation}
\exp\biggpar{-(C_++C_-)\frac{\Gamma(3-\ga)}{\ga(1-\ga)}
\Bigpar{\cos\frac{\pi\ga}2-\ii\sgn(t)\frac{C_+-C_-}{C_++C_-} 
  \sin\frac{\pi\ga}2}|t|^\ga+\ii bt}
\end{equation}
if $\ga\neq1$
and
\begin{equation}
\exp\biggpar{-(C_++C_-)
\Bigpar{\frac{\pi}2+\ii\sgn(t)\frac{C_+-C_-}{C_++C_-} \log|t|}|t|+\ii bt}  
\end{equation}
if $\ga=1$.
This is \eqref{chf} with
\eqref{chf2a}--\eqref{chf2b} and $\gd=b$.
This proves (i) and (iii), and (iv) follows from (iii) by \eqref{erika}.

Finally, (ii) follows directly from \eqref{chf}.
\end{proof}

\begin{remark}\label{Rstrict}
If $1<\ga\le2$, then $\gd$ in \eqref{chf} equals the mean $\E
\XXXX\ga\gam\gb\gd$.
In particular, \eqref{tchfb} shows that
for $\ga>1$, a stable distribution is strictly stable if and
only if its expectation vanishes.
\end{remark}

\begin{remark}
If $\XXXx\ga(\gb)\sim \SSSx\ga(\gb)=\SSSS\ga1\gb0$, then,
for $\gam>0$ and $\gd\in\bbR$,
  \begin{equation}\label{lin}
	\gam \XXXx\ga(\gb)+\gd\sim 
	\begin{cases}
\SSSS\ga\gam\gb\gd, & \ga\neq1, \\
\SSSS\ga\gam\gb{\gd-\tfrac2\pi \gb\gam\log\gam},&	 \ga=1.
	\end{cases}
  \end{equation}
Thus,  $\gam$ is a scale parameter and $\gd$ a location parameter; $\gb$ is
  a skewness parameter, and $\ga$ and $\gb$  together determine the
  shape of   the distribution. 
\end{remark}

\begin{remark}\label{Rlinx}
  More generally, if 
$X\sim \SSSS\ga\gam\gb\gd$, then,
for $a>0$ and $d\in\bbR$,
  \begin{equation}\label{linx}
	a X+d\sim 
	\begin{cases}
\SSSS\ga{a\gam}\gb{a\gd+d}, & \ga\neq1, \\
\SSSS\ga{a\gam}\gb{a\gd+d-\tfrac2\pi \gb\gam a\log a},&	 \ga=1.
	\end{cases}
  \end{equation}
\end{remark}

\begin{remark}\label{R-}
  If $X\sim \SSSS\ga\gam\gb\gd$, then $-X \sim \SSSS\ga\gam{-\gb}{-\gd}$.
In other words,
\begin{align}\label{r-}
  -\XXXX\ga\gam\gb\gd \eqd \XXXx\ga(\gam,-\gb,-\gd).
\end{align}
In particular, $X$ has a symmetric stable distribution if and only if
$X\sim \SSSS\ga\gam 00$ for some $\ga\in(0,2]$ and $\gam>0$.
\end{remark}

We may simplify expressions like \eqref{chf} by considering only $t\ge0$
(or $t>0$); this is sufficient because of the general formula
\begin{equation}\label{ch+-}
  \gf(-t)=\overline{\gf(t)}
\end{equation}
for any \chf.
We use this in our next statement, which is an immediate consequence of
\refT{Tchf}. 

\begin{corollary}
  \label{CJ}
Let $0<\ga\le2$.
A distribution is strictly stable if and only if it has a \chf{}
\begin{equation}\label{jepp1}
  \gf(t)=\exp\bigpar{-(\kk-\ii\tau)t^\ga},
\qquad t\ge0,
\end{equation}
where $\kk>0$ and $|\tau|\le\kk|\tan\piaq|$; furthermore, a strictly stable
distribution exists for any such $\kk$ and $\tau$.
(For $\ga=1$, $\tan\piaq=\infty$, so any real $\tau$ is possible.
For $\ga=2$, $\tan\piaq=0$, so necessarily $\tau=0$.)

The distribution $\SSSS\ga\gam\gb0$ ($\ga\neq1$) or
$\SSSx1(\gam,0,\gd)$ ($\ga=1$) satisfies \eqref{jepp1} with 
\begin{equation}\label{jepp1a}
\kk=\gam^\ga
\qquad\text{and}\qquad
  \tau=
  \begin{cases}
\gb \kk\tan\piaq,
& \ga\neq1,\\
\gd,&\ga=1.
  \end{cases}
\end{equation}
Conversely, if \eqref{jepp1} holds, then the distribution is
\begin{equation}\label{jepp1b}
  \begin{cases}
\SSSS\ga\gam\gb0 \text{ with }  
\gam=\kk\xga,\, \gb=\frac{\tau}{\kk}\cot\piaq,& \ga\neq1\\
\SSSx1(\kk,0,\tau),& \ga=1.
  \end{cases}
\end{equation}
\nopf
\end{corollary}

\begin{remark}
For a strictly stable \rv,
  another way to write the \chf{} \eqref{chf} or \eqref{jepp1}
is
  \begin{equation}\label{joh}
	\gf(t)=\exp\Bigpar{-\glx e^{\ii\sgn(t)\pi\tgam/2}|t|^\ga},
  \end{equation}
with $\glx>0$ and $\tgam$ real (with $|\tgam|\le1$; see further below).
A comparison with \eqref{chf} and \eqref{jepp1a} shows that
\begin{align}\label{joha}
  \glx\cos\pigq&=\kk=\gam^\ga,\\
\tan\pigq&=
-\frac{\tau}{\kk}=
\begin{cases}
-\gb\tan\piaq, & \ga\neq1,\\
-\frac{\gd}{\gam^\ga}, & \ga=1.  
\end{cases}
\label{johb}
\end{align}
If $0<\ga<1$, we have $0<\tan\piaq<\infty$ and $|\tgam|\le\ga$, while if
$1<\ga<2$, then $\tan\piaq<0$ and
$\tan\pigq=\gb\tan\frac{\pi(2-\ga)}{2}$ with $0<\pi(2-\ga)/2<\pi/2$; hence
$|\tgam|\le 2-\ga$. 
Finally, for $\ga=1$, we have $|\tgam|<1$, and for $\ga=2$ we have $\tgam=0$.
These ranges for $\tgam$ are both necessary and sufficient, except that
for $\ga=1$,
$\tgam=\pm1$ is possible in \eqref{joh}, but yields a degenerate
distribution $X=-\tgam \glx$.
Summarising, we have the ranges, excluding the degenerate case just mentioned, 
\begin{align}\label{tgamr}
  \begin{cases}
  |\tgam|\le  \ga,& 0<\ga<1,\\
  |\tgam|<  1,& \ga=1,\\
  |\tgam|\le  2-\ga,& 1< \ga\le2.
  \end{cases}
\end{align}
For $\ga\neq1,2$,
note the special cases 
\begin{align}\label{johc0}
\gb=0&\iff\tgam=0,\\
 \gb=1&\iff\tgam=
 \begin{cases}
-\ga,& 0<\ga<1,
\\   
2-\ga,& 1<\ga<2.
 \end{cases}\label{johc}
\\
 \gb=-1&\iff\tgam=
 \begin{cases}
\ga,& 0<\ga<1,
\\   
\ga-2,& 1<\ga<2.
 \end{cases}\label{johd}
\end{align}
\end{remark}

\begin{remark}\label{R11}
  For $\ga=1$, the general 1-stable \chf{} \eqref{chf} may be written,
  similarly to \eqref{jepp1},
\begin{equation}\label{jepp11}
  \gf(t)=\exp\bigpar{-(\kk-\ii\tau)t-\ii bt\log t},
\qquad t>0,
\end{equation}
where $\kk=\gam$, $\tau=\gd$ and $b=\frac2\pi \gb \gam$.
(Thus, $|b|\le 2\kk/\pi$.)
\end{remark}

\subsection{Positive and spectrally positive stable distributions}

\begin{definition}\label{Dspectral>}
  A stable distribution (or random variable)
is \emph{spectrally positive} if its \Lm{} is
  concentrated on $\oooy$, \ie, 
  \begin{equation}
\dd\gL(x) = cx^{-\ga-1}\dd x,\qquad x>0, 	
  \end{equation}
for some $c>0$ and
  $\ga\in(0,2)$. 
By \eqref{tsL} and \eqref{chf3b}, this is equivalent to $c_-=0$ and to
$\gb=1$, see also \eqref{johc}.

Similarly,
a stable distribution (or random variable)
is \emph{spectrally negative} if its \Lm{} is
  concentrated on $(-\infty,0)$.
\end{definition}
Thus, $X$ is spectrally negative if and only if $-X$ is spectrally positive.
(For this reason, we mainly consider the spectrally positive case.)

\begin{theorem}\label{Tstrict+}
A strictly stable distribution is spectrally
positive if and only if it is of the form
$\SSSS\ga\gam10$ with $\ga\neq1$.

Equivalently,
  a strictly stable distribution with \chf{} \eqref{jepp1} is spectrally
positive if and only if $\ga\neq1$ and $\tau=\kk\tan\piaq$.
\end{theorem}

\begin{proof}
 This follows from \refC{CJ}, taking $\gb=1$ in \eqref{jepp1a};
note that by \eqref{jepp1b}, there is no spectrally positive strictly
1-stable  distribution.
\end{proof}

\begin{theorem}
  \label{Texp}
Let\/ $0<\ga<2$. An $\ga$-stable random variable $X\sim \SSSx\ga(\gam,\gb,\gd)$
has finite Laplace 
transform $\E e^{-tX}$ for $t\ge0$ if and only if it is spectrally positive,
\ie, if $\gb=1$, and then
\begin{equation}\label{texp}
\E e^{-tX}=
 \begin{cases}
\exp\Bigpar{-\frac{\gam^\ga}{\cos\frac{\pi\ga}{2}}t^\ga-\gd t}, &\ga\neq1,
\\
\exp\Bigpar{\frac2{\pi}\gam t\log t-\gd t}
, & \ga=1,
 \end{cases}
\end{equation}  
Moreover, then
\eqref{texp} holds for every complex $t$ with $\Re t\ge0$.
\end{theorem}
\begin{proof}
  The condition for finiteness follows by \refT{Tidexp} and \eqref{tsL},
  together with \refD{Dspectral>}. When this holds, 
the \rhs{} of \eqref{texp} is a continuous function of $t$ in the closed
right half-plane $\Re t\ge0$, which is analytic in the open half-plane $\Re
t>0$. The same is true for the \lhs{} by \refT{Tidexp}, and the two functions
are equal on the imaginary axis $t=\ii s$, $s\in\bbR$ by \eqref{chf} and a
simple calculation. By uniqueness of analytic continuation, 
\eqref{texp} holds for every complex $t$ with $\Re t\ge0$.
\end{proof}

\begin{theorem}
  \label{T+}
An stable random variable $X\sim \SSSx\ga(\gam,\gb,\gd)$
is positive, \ie{} $X>0$ \as, 
if and only if 
$0<\ga<1$, $\gb=1$ and $\gd\ge0$.
Consequently,
the positive strictly stable random variables are
$\XXXX\ga\gam10$ with $0<\ga<1$.
\end{theorem}

\begin{proof}
$X>0$ \as{} if and only if the Laplace transform $\E e^{-tX}$ is finite for all
  $t\ge0$ and $\E e^{-tX}\to0$ as \ttoo.
Suppose that this holds.
We cannot have $\ga=2$, since then $X$ would be normal and therefore not
positive; thus \refT{Texp} 
applies and shows that $\gb=1$. Moreover, \eqref{texp} holds. If $1<\ga<2$
or $\ga=1$, then the \rhs{} of \eqref{texp} tends to infinity as \ttoo{},
which is a contradiction; hence $0<\ga<1$, and then \eqref{texp} again shows
that $\gd\ge0$. 

The converse is immediate from \eqref{texp}.
\end{proof}

\begin{corollary}\label{CT+}
Let $X$ be a stable random variable.
Then, $X>0$ \as{} if and only if $X=Y+\gd$ where 
$\gd\ge0$ and
$Y$ is  spectrally positive strictly $\ga$-stable with $0<\ga<1$.
\nopf
\end{corollary}

The following examples are the two most important cases of \refT{Texp}.

\begin{example}\label{E+}
  If $0<\ga<1$ and $\gld>0$, then
$X\sim \SSSx\ga(\gam,1,0)$ with $\gam\=\bigpar{\gld\cos\piaq}\xga$ 
is a positive strictly stable
\rv{} with the Laplace transform (extended by analyticity)
  \begin{equation}\label{e+}
\E e^{-tX}	=\exp\bigpar{-\gld t^\ga},
\qquad\Re t\ge0.
  \end{equation}
Note that we have $\tgam=-\ga$ by \eqref{johc}.
\end{example}

\begin{example}\label{E+2}
  If $1<\ga<2$ and $\gld>0$, then
$X\sim \SSSx\ga(\gam,1,0)$ with $\gam\=\bigpar{\gld|\cos\piaq|}\xga$ 
is a spectrally positive strictly stable
  \rv{} with the Laplace transform (extended by analyticity)
  \begin{equation}\label{e+2}
\E e^{-tX}	=\exp\bigpar{\gld t^\ga},
\qquad\Re t\ge0.
  \end{equation}
Note that in this case $\cos\piaq<0$. Note also that 
$\E e^{-tX}\to\infty$ as \ttoo, which shows that $\P(X<0)>0$.
\end{example}

\subsection{Other parametrisations}

Our notation $\SSSS\ga\gam\gb\gd$  is in accordance with \eg{}  
\citet[Definition 1.1.6 and page 9]{SamTaq}.
(Although they use the letters $\SSSS\ga\gs\gb\mu$.)
\citet{Nolan} uses 
the notation $S(\ga,\gb,\gam,\gd;1)$;
he also defines
$  S(\ga,\gb,\gam,\gd_0;0)\=S(\ga,\gb,\gam,\gd_1;1)$
where
\begin{equation}
\gd_1\=
\begin{cases}
  \gd_0-\gb\gam \tan\frac{\pi\ga}2, & \ga\neq1, \\
  \gd_0-\frac2\pi \gb\gam \log\gam, & \ga=1.
\end{cases}
\end{equation}
(Note that our $\gd=\gd_1$.)
This parametrisation has the advantage that the distribution
$  S(\ga,\gb,\gam,\gd_0;0)$ is a continuous function of all four parameters.
Note also that 
$  S(\ga,0,\gam,\gd;0)=S(\ga,0,\gam,\gd;1)$, and that when $\ga=1$,
\eqref{lin} becomes $\gam X_1(\gb)+\gd\sim S(1,\gam,\gb,\gd;0)$.
Cf.\ the related parametrisation in \cite[Remark 1.1.4]{SamTaq}, which uses
\begin{equation}
  \mu_1=
\begin{cases}
  \gd_1+\gb\gam^\ga \tan\frac{\pi\ga}2
=  \gd_0+\gb(\gam^\ga-\gam) \tan\frac{\pi\ga}2, & \ga\neq1, \\
  \gd_1
\phantom{{}+\gb\gam^\ga \tan\frac{\pi\ga}2}
=  \gd_0-\frac2\pi \gb\gam \log\gam, & \ga=1;
\end{cases}
\end{equation}
again the distribution is a continuous function of $(\ga,\gb,\gam,\mu_1)$.

\citet{Z} uses three different parametrisations, with parameters denoted
$(\ga,\gb_\zz,\gam_\zz,\gl_\zz)$, where $\zz\in\set{A,B,M}$; these are
defined by 
writing the \chf{} \eqref{chf} as
\begin{align}
\gf(t)&=\exp\bigpar{\gl_A\bigpar{\ii t\gam_A-|t|^\ga+\ii t\go_A(t,\ga,\gb_A)}}
\label{chfZA}\\
&=\exp\bigpar{\gl_M\bigpar{\ii t\gam_M-|t|^\ga+\ii t \go_M(t,\ga,\gb_M)}}
\label{chfZM}\\
&=\exp\bigpar{\gl_B\bigpar{\ii t\gam_B-|t|^\ga \go_B(t,\ga,\gb_B)}},
\label{chfZB}
\end{align}
where
\begin{align}
\go_A(t,\ga,\gb)&\=
  \begin{cases}
	|t|^{\ga-1}\gb\tan\piaq, & \ga\neq1, \\
	-\gb\frac2\pi\log|t|, & \ga=1;
  \end{cases}
\\
\go_M(t,\ga,\gb)&\=
  \begin{cases}
	(|t|^{\ga-1}-1)\gb\tan\piaq, & \ga\neq1, \\
	-\gb\frac2\pi\log|t|, & \ga=1;
  \end{cases}
\\
\go_B(t,\ga,\gb)&\=
  \begin{cases}
    \exp\bigpar{-\ii\frac\pi2 \gb K(\ga)\sgn t}, & \ga\neq1, \\
	\frac\pi2+\ii\gb\log|t|\sgn t, & \ga=1,
  \end{cases}  
\end{align}
with $K(\ga)\=\ga-1+\sgn(1-\ga)$, \ie,
\begin{equation}\label{Kga}
  K(\ga)\=
  \begin{cases}
	\ga, & 0<\ga<1, \\
\ga-2, & 1<\ga\le2.
  \end{cases}
\end{equation}
The ranges of the parameters are, in all three cases
$\bullet\in\set{A,B,M}$,
\begin{align}
  0<\ga\le2,&&& -1\le\gb_\zz\le1,&&-\infty<\gam_\zz<\infty,&&0<\gl_\zz<\infty.
\end{align}
If $\ga=2$, we take $\gb_\zz=0$.

Here $\ga\in(0,2]$ is the same in all parametrisations
and, 
with $\gb,\gam,\gd$ is as in \eqref{chf}, 
\begin{align}
  \gb_A&=\gb_M=\gb,\label{gbAM}
\\
\gam_A&=\gd/\gam^\ga, \label{ggA}\\
\gam_M&=
\mu_1/\gam^\ga=
\begin{cases}
  \gam_A+\gb\tan\piaq, & \ga\neq1, \\
  \gam_A, & \ga=1,
\end{cases}
\\
  \gl_A&=\gl_M=\gam^\ga, \label{glAM}   
\intertext{and, for $\ga\neq1$,}
\tan\Bigpar{\gb_B\frac{\pi K(\ga)}2}
&= \gb_A\tan\piaq = \gb\tan\piaq,  \label{gbB}
\\
\gam_B&=\gam_A\cos\Bigpar{\gb_B\frac{\pi K(\ga)}2},\label{ggB}
\\
\gl_B&=\gl_A\bigm/\cos\Bigpar{\gb_B\frac{\pi K(\ga)}2},\label{glB}
\intertext{while for $\ga=1$,} 
\gb_B&=\gb_A=\gb, \label{gbB1}\\
\gam_B&=\frac{\pi}2\gam_A=\frac{\pi\gd}{2\gam}, \label{ggB1}\\
\gl_B&=\frac2\pi\gl_A=\frac{2\gam}\pi.\label{glB1}
\end{align}
Note that, for any $\ga$, and every
$\zz\in\set{A,B,M}$,
\begin{align}
  \label{gbbb}
\gb_\zz=0\iff\gb=0 
\qquad\text{and}\qquad \gb_\zz=\pm1\iff\gb=\pm1, 
\end{align}
and that for each fixed $\ga$, the mapping
$\gb=\gb_A\mapsto\gb_B$ is an increasing homeomorphism of $[-1,1]$ onto itself.

In the strictly stable case, \citet{Z} also uses
  \begin{equation}\label{chfZC}
	\gf(t)=\exp\Bigpar{-\gl_C e^{-\ii\sgn(t)\pi\ga\gth/2}|t|^\ga},
  \end{equation}
which is the same as \eqref{joh} with
\begin{align}
\gl_C&=\glx \label{glC-gl}
\\
\gth&=-\tgam/\ga; \label{gth-tgam}
\end{align}
thus the ranges of the parameters are
(excluding the  case $\ga=1$ and $\gth=\pm1$, 
which is possible in \eqref{chfZC}, but degenerate)
\begin{align}\label{gthra}
&  \begin{cases}
|\gth|   \le 1,&\ga<1,\\
|\gth|   < 1,&\ga=1,\\
|\gth|\le2/\ga-1, &\ga>1,
  \end{cases}
\\\label{glCra}
&\quad 0<\gl_C<\infty.
\end{align}
We  have
\begin{align}
\gth&=
  \begin{cases}
\gb_B\frac{K(\ga)}{\ga}, & \ga\neq1 ,
\label{gth-gbB}
\\
\frac2\pi\arctan\bigpar{\frac{2\gam_B}{\pi}}, &\ga=1.	
  \end{cases}
\\\label{glC-glB}
\gl_C&=
  \begin{cases}
\gl_B, & \ga\neq1 ,\\
\gl_B\bigpar{\pi^2/4+\gam_B^2}\qq, &\ga=1.	
  \end{cases}
\end{align}
\citet{Z} uses in the strictly stable case also
the parameters $\ga,\rho,\gl_C$ where 
\begin{equation}
  \label{rho}
\rho\=\frac{1+\gth}2.
\end{equation}
Thus the range of $\rho$ is
\begin{align}\label{rhor}
  \begin{cases}
    0\le\rho\le1,&\ga<1,
\\
    0<\rho<1,&\ga=1,
\\
1-1/\ga\le\rho\le1/\ga,&\ga>1.
\end{cases}
\end{align}


\citet{Z} uses 
$Y(\ga,\gb_\zz,\gam_\zz,\gl_\zz)=Y_\zz(\ga,\gb_\zz,\gam_\zz,\gl_\zz)$,
where again $\zz\in\set{A,B,M}$, 
as a notation for a random
variable with the \chf{} \eqref{chfZA}--\eqref{chfZB}; the parameters
$\gam_\zz$ and $\gl_\zz$ may be omitted when
$\gam_\zz=0$ and $\gl_\zz=1$.
The distribution is a continuous function of the parameters
$(\ga,\gb_M,\gam_M,\gl_M)$. 
(The representations $A$ and $B$ are discontinuous at $\ga=1$.)
Similarly,  a random
variable with the \chf{} \eqref{chfZC}
is denoted 
$Y(\ga,\gth,\gl_C)=Y_C(\ga,\gth,\gl_C)$, where
$\gl_C$ may be omitted when
$\gl_C=1$.
We use $\YY_\zz(\dots)$ for the distribution of
$Y_\zz(\dots)$.

The parameter $\rho$ has a natural interpretation.
(See \refT{Tmom} for a generalization.)

\begin{theorem}\label{T>0}
  For a strictly stable random variable $Y_C(\ga,\gth,\gl)$,
  \begin{align}\label{t>0}
\P\bigsqpar{Y_C(\ga,\gth,\gl)>0}
= \rho=\frac{1+\gth}2.    
  \end{align}
\end{theorem}
\begin{proof}
  See, e.g., \cite[Theorem 2.6.3]{Z} (in the special case $s=0$).
\end{proof}

\begin{corollary}\label{C>0}
The strictly stable random variable $Y_C(\ga,\gth,\gl)$
is positive
$\iff$ $\ga<1$ and $\rho=1$
$\iff$ $\ga<1$ and $\gth=1$.

Similarly,
$Y_C(\ga,\gth,\gl)$
is negative
$\iff$ $\ga<1$ and $\rho=0$
$\iff$ $\ga<1$ and $\gth=-1$.
\end{corollary}
\begin{proof}
  By \eqref{t>0} and \eqref{rhor}.
\end{proof}

Using \refT{T>0},
\eqref{gth-gbB} and \eqref{gbB}, the probability that a strictly
stable
random variable is positive can be expressed in
$\ga$ and $\gb_B$ or $\gb$ when $\ga\neq1$,
and in $\gam_B$ or (using also \eqref{ggB})
$\gam$ and $\gd$ when $\ga=1$.
In particular, this yields
\begin{align}
  \P\bigsqpar{\XXXx\ga(\gam,\gb,0)>0}& 
=\frac12+\frac{1}{\ga\pi}\arctan\bigpar{\gb\tan\frac{\pi\ga}{2}}
,\qquad\ga\neq1,
\\
  \P\bigsqpar{\XXXx1(\gam,0,\gd)>0}& 
=\frac12+\frac{1}{\pi}\arctan\Bigpar{\frac{\gd}{\gam}}.
\end{align}

\begin{example}
  \label{E+Z}
By \refC{C>0}, the positive strictly stable random variable 
$\XXXX\ga\gam10$
in \refT{T+}
can also be described as
$Y_C(\ga,1,\gl)$; here necessarily $0<\ga<1$.
This random variable has, using 
\eqref{glC-gl}--\eqref{gth-tgam}, \eqref{rho},
\eqref{johc} and
\eqref{joha},
the parameters
\begin{align}\label{e+z}
\gb=1,\quad& \gth=1,\quad \rho=1,\quad  \tgam=-\ga,\quad
\gam^\ga=\gl\,{\cos\piaq},
\end{align}
and, by \refT{Texp}, the Laplace transform
\begin{equation}
  \label{julie}
\E e^{-tY_C(\ga,1,\gl)}=e^{-\gl t^\ga},
\qquad t\ge0.
\end{equation}
For $0<\ga<1$,
$Y_C(\ga,1,\gl)$ is thus the random variable in
\refE{E+}.
\end{example}

We have a similar result for the extreme values in 
\eqref{gthra} and \eqref{rhor} also for the case
$\ga>1$. (The Gaussian case $\ga=2$ is trivial; then necessarily $\gth=0$
and $\rho=1/2$ by \eqref{gthra} and \eqref{rhor}.)

\begin{theorem}\label{T+rho}
  Let\/ $1<\ga\le2$.
The strictly stable random variable $Y_C(\ga,\gth,\gl)$
is spectrally positive
$\iff$  $\rho=1-1/\ga$
$\iff$ $\gth=1-2/\ga$.

Similarly,
$Y_C(\ga,\gth,\gl)$
is spectrally negative
$\iff$  $\rho=1/\ga$
$\iff$ $\gth=2/\ga-1$.
\end{theorem}
Note that when $1<\ga<2$,
thus $\gth<0$ in the spectrally positive case, and $\gth>0$ in the
spectrally negative case.

\begin{proof}
  By  \refT{Tstrict+},
\eqref{gbbb}, \eqref{gth-gbB}, \eqref{Kga} and \eqref{rho}.
\end{proof}

\begin{example}\label{E+2Z}  
Let $1<\ga<2$.
By \refT{T+rho},
the spectrally positive strictly stable random variable 
$\XXXX\ga\gam10$
in \refT{Tstrict+}
can also be described as
$Y_C(\ga,\gth,\gl)$ with
$\gth=1-2/\ga$. 
This random variable has, using \eqref{johc} and \eqref{joha},
\begin{align}
\gb=1,&&& \gth=1-\frac{2}{\ga},&& \rho=1-\frac{1}\ga,&&  
\tgam=2-\ga,&& 
\gam^\ga=\gl\bigabs{\cos\piaq},
\end{align}
and, by \refT{Texp}, 
the Laplace transform
\begin{equation}
  \label{julie2}
\E e^{-tY_C(\ga,1,\gl)}=e^{\gl t^\ga},
\qquad t\ge0.
\end{equation}
For $1<\ga<2$, $Y_C(\ga,1-2/\ga,\gl)$ is thus the random variable in
\refE{E+2}.

By \eqref{t>0}, we have
\begin{align}
\P\bigsqpar{Y_C(\ga,1-2/\ga,\gl)>0}=\rho=1-\frac{1}{\ga}.
\end{align}
\end{example}

\section{Stable densities}
A stable distribution has by \eqref{chf} a \chf{} that decreases rapidly as
$t\to\pm\infty$, and thus the distribution has a density that is infinitely
differentiable. 

In the case $\ga<1$ and $\gb=1$, $\SSSS\ga\gam\gb\gd$ has support 
$[\gd,\infty)$
and in the case $\ga<1$ and $\gb=-1$, $\SSSS\ga\gam\gb\gd$ has support 
$(-\infty,\gd]$; in all other cases the support is the entire real line.
Moreover, the density function is strictly positive in the interior of the
support, se \citet[Remark 2.2.4]{Z}.

\citet[Section XVII.6]{FellerII} lets, for $\ga\neq1$, 
$p(x;\ga,\tgam)$
denote the density of the stable distribution with \chf{} \eqref{joh} with
$\glx=1$. A stable \rv{} with the \chf{} \eqref{joh} thus has the density
function
$\glx\xgaw p(\glx\xgaw x;\ga,\tgam)$.
The density of a random variable $\XXXx\ga(\gam,\gb,\gd)$ with $\ga\neq1$
is thus given by
\begin{equation}\label{sd1}
 \glx\xgaw p\bigpar{\glx\xgaw (x-\gd);\ga,\tgam}, 
\end{equation}
with $\glx$ and $\tgam$ given by \eqref{joha}--\eqref{johb}.
(Cf.\ \refR{Rlinx}.)
By \refR{R-}, we have also 
\begin{align}
  \label{sd-}
p(-x;\ga,\tgam)=p(x;\ga,-\tgam).
\end{align}

\citet{Z} uses 
$g_\zz(x;\ga,\gb_\zz,\gam_\zz,\gl_\zz)$   
for the density of the random variable
$Y_\zz(\ga,\gb_\zz,\gam_\zz,\gl_\zz)$ with \chf{} \eqref{chfZA}--\eqref{chfZB},
and 
$g_\zz(x;\ga,\gb_\zz)$ for the special case $\gam_\zz=0$, $\gl_\zz=1$;
the index $\zz\in\set{A,M,B}$ is often omitted (and often, but not always,
taken as $B$); furthermore,
$g(x;\ga,\gth)=g_C(x;\ga,\gth)$ is used for the density of the \rv{}
$Y_C(\ga,\gth)$ with \chf{} \eqref{chfZC} with $\gl_C=1$.
Thus, for $\ga\neq1$, see \eqref{glC-gl}--\eqref{gth-tgam},
\begin{equation}
  \label{gC-p}
g_C(x;\ga,\gth)=p(x;\ga,-\ga\gth).
\end{equation}
By \eqref{chfZC}, we have also, 
in analogy with \eqref{sd-} (but now for all $0<\ga\le2$), 
\begin{align}\label{gC-}
g_C(-x;\ga,\gth)=g_C(x;\ga,-\gth).  
\end{align}

\citet[Lemma XVII.6.1]{FellerII} 
and
\citet[(2.4.8) and (2.4.6)]{Z} 
give the following series expansions for $p(x;\ga,\tgam)$ and
$g_C(x;,\ga,\gth)$, 
repsectively;
the latter using $\rho\=(1+\gth)/2$ as in \eqref{rho}.
These expansions are equivalent by \eqref{gC-p}.

\begin{theorem}\label{Tp}
  \begin{romenumerate}[-10pt]
  \item 
If\/ $0<\ga<1$ and $x>0$, then
\begin{align}\label{pgam<}
p(x;\ga,\tgam)&
=\frac1{\pi x}
\sumki\frac{\Gamma(k\ga+1)}{k!}(-x^{-\ga})^k\sin\frac{k\pi}{2}(\tgam-\ga),
\\\label{pgC<}
g_C(x;\ga,\gth)&
=\frac1{\pi}
\sumki(-1)^{k-1}\frac{\Gamma(k\ga+1)}{k!}\sin(\pi k\rho\ga)x^{-k\ga-1}
.\end{align}

For $x<0$, use \eqref{pgam<}--\eqref{pgC<} together with  \eqref{sd-}
and \eqref{gC-}.

  \item 
If\/ $1<\ga\le2$ and $x\in(-\infty,\infty)$, 
then
\begin{align}\label{pgam>}
p(x;\ga,\tgam)&
=\frac1{\pi x}
\sumki\frac{\Gamma(k/\ga+1)}{k!}(-x)^k\sin\frac{k\pi}{2\ga}(\tgam-\ga),
\\\label{pgC>}
g_C(x;\ga,\gth)&
=\frac1{\pi}
\sumki(-1)^{k-1}
\frac{\Gamma(k/\ga+1)}{k!}\sin(\pi k\rho)x^{k-1}
.\end{align}
  \end{romenumerate}
\nopf
\end{theorem}

\begin{remark}
The symmetry relations   \eqref{sd-} and \eqref{gC-}
are valid for all $\ga$, but not needed
in \refT{Tp} for $\ga>1$, since then \eqref{pgam>}--\eqref{pgC>} hold for
all real $x$ 
(with the obvious interpretation of \eqref{pgam>} for $x=0$). 
It can easily by verified directly that
\eqref{pgam>}--\eqref{pgC>} satisfy \eqref{sd-} and \eqref{gC-}. 
\end{remark}

\begin{example}
  The case $\ga=2$ is simple; then $\tgam=0$, $\gth=0$ and $\rho=1/2$ by
  \eqref{tgamr}, \eqref{gthra} and \eqref{rhor}, and the \chf{} \eqref{chfZC}
shows that $Y_C(2,0)\sim N(0,2)$. Hence,
\begin{align}
  p(x;2,0) = g_C(x;2,0)=\frac{1}{2\sqrt\pi}e^{-x^2/4},
\end{align}
which indeed has the series expansions \eqref{pgam>}--\eqref{pgC>}.
\end{example}

In particular, if $1<\ga\le2$, then \eqref{pgam>} yields
\begin{equation}
  p(0;\ga,\tgam)=\frac1\pi\Gamma(1+1/\ga)\sin\frac{\pi(\ga-\tgam)}{2\ga}.
\end{equation}
In the special case $1<\ga<2$ and
$\gb=1$ we have $\tgam=2-\ga$ by \eqref{johc} and 
\begin{equation}
  \begin{split}
  p(0;\ga,2-\ga)&=\frac1\pi\Gamma(1+1/\ga)\sin\frac{\pi(\ga-1)}{\ga}
=\frac1\pi\Gamma(1+1/\ga)\sin\frac{\pi}{\ga}	
\\&
=\frac{\Gamma(1+1/\ga)}{\Gamma(1/\ga)\Gamma(1-1/\ga)}
=\frac{1}{\ga\Gamma(1-1/\ga)}
=\frac{1}{|\Gamma(-1/\ga)|}.
  \end{split}
\end{equation}

For $1<\ga<2$, the distribution $\SSSx\ga(\gam,1,0)$ thus has, by 
\eqref{sd1} and \eqref{joha}, the density at $x=0$
\begin{equation}\label{sd2}
  \glx\xgaw p(0;\ga,2-\ga)=
\frac{\glx\xgaw}{|\Gamma(-1/\ga)|}
=\gam\qw\bigabs{\cos\piaq}\xga |\Gamma(-1/\ga)|\qw.
\end{equation}

\subsection{The case $\ga=1$}

The case $\ga=1$ was omitted in \refT{Tp}, since there is no similar simple
formula, except when $\gb=0$.
However, we have the following power series expansion for $\ga=1$ and
$\gb\neq0$, given by \citet{Z}.
\begin{theorem}\label{T1}
  Let $\ga=1$.
  \begin{romenumerate}
  \item \label{T1=}
If $\gb=0$,    
then $\SSSS1\gam0\gd$ has the density function
\begin{align}\label{t1a}
  \frac{\gam/\pi}{(x-\gd)^2+\gam^2},
\qquad-\infty<x<\infty
.\end{align}
\item \label{T1>}
If $\gb>0$, then 
$Y_B(1,\gb,0,1) = \XXXx1\bigpar{\frac{\pi}2,\gb,0}$ has the density function 
\begin{align}\label{t1b}
  g_B(x;1,\gb)=\frac{1}{\pi}\sumn (-1)^n c_n x^n  
,\end{align}
with
\begin{align}\label{t1c}
  c_n :=
\frac{1}{n!}\intoo e^{-\gb u\log u}\sin \bigsqpar{(1+\gb)\frac{\pi}2u}u^{n} \dd u
\end{align}
\item \label{T1<}
If $\gb<0$, then 
$Y_B(1,\gb,0,1) = \XXXx1\bigpar{\frac{\pi}2,\gb,0}$ has the density
function 
\begin{align}\label{t1d}
  g_B(x;1,\gb)=g_B(-x;1,-\gb),
\end{align}
which is given by \eqref{t1b}.
  \end{romenumerate}
\end{theorem}
\begin{proof}
  \pfitemref{T1=}
This well-known formula follows directly by Fourier inversion of the
characteristic function $\gf(t)=e^{-\gam|t|+\ii\gd t}$.

\pfitemref{T1>}
Note first that
if $\ga=1$, then \eqref{gbB1}--\eqref{glB1} show that $\gb_B=\gb$,
$\gam_B=0\iff\gd=0$, and $\gl_B=1\iff\gam=\pi/2$. Hence,
$Y_B(1,\gb,0,1) = \XXXx1\bigpar{\frac{\pi}2,\gb,0}$ as asserted.

The expansion \eqref{t1b}--\eqref{t1c} is \cite[(2.4.7)]{Z} (with our $c_n$
equal to $(n+1)b_{n+1}$ there).

\pfitemref{T1<}
This follows by  \eqref{r-}.
\end{proof}

\subsection{Analyticity}
The density of any stable distribution $\SSSS\ga\gam\gb\gd$
is, as said above, infinitely differentiable.
Moreover, it is easy to see from \refTs{Tp} and \ref{T1}
that this density is real analytic for
$x\neq\gd$. At $x=\gd$, the situation differs for $\ga<1$ and $\ga\ge1$, as
shown 
by the following result.
\begin{theorem}\label{TA}
Consider the density  $p(x)$ of  $X\sim \SSSx\ga(\gam,\gb,\gd)$.
\begin{romenumerate}
\item   \label{ta>}   
If $\ga\ge1$, then $p(x)$ is real analytic on $(-\infty,\infty)$.
\item   \label{ta<}
If $\ga<1$, then p(x) is real analytic on $\bbR\setminus\set{\gd}$,
but not at $\gd$ (although it is infinitely differentiable there too).
\end{romenumerate}
\end{theorem}

\begin{proof}
\pfitemref{ta>}
For $\ga\neq1$,
by \eqref{sd1}, it suffices to consider $p(x;\ga,\tgam)$,
and the analyticity  follows from \eqref{pgam>}.

For $\ga=1$, analyticity follows from \eqref{t1a}, \eqref{t1b} or
\eqref{t1d} (depending on $\gb$), together with a liear change of variable.

\pfitemref{ta<}
Again, by \eqref{sd1} it suffices to consider $p(x;\ga,\tgam)$,
and thus $\gd=0$.
The analyticity for $x>0$ follows from \eqref{pgam<}, and then for $x<0$
from \eqref{sd-}. These also show that $p(x)=p(x;\ga,\tgam)$ extends to an
analytic function $p(z)$ in each of the half planes $\Re z<0$ and $\Re z>0$,
with  
\begin{align}\label{jesper}
|p(z)|=O\bigpar{|z|^{-1-\ga}},  \qquad |z|\ge 1.  
\end{align}
Suppose that $p$ is real analytic also at $x=0$. Then $p$ would extend to an
analytic function in a neighbourhood of $0$, and thus the extensions would
combine to an analytic extension in a strip $|\Im z|<2\eps$ for some
$\eps>0$.
The characteristic function $\gf(t)$
then would be given by, 
by a shift of the line of integration
using Cauchy's integral formula and the bound \eqref{jesper},
\begin{align}
  \gf(t) =\intoooo e^{\ii tx}p(x)\dd x
=
\intoooo e^{\ii t(x+\ii\eps)}p(x+\ii\eps)\dd x,
\qquad t\in\bbR,
\end{align}
and thus, by \eqref{jesper} again,
\begin{align}
\bigabs{\gf(t)}
\le
 e^{-\eps t}\intoooo\bigabs{p(x+\ii\eps)}\dd x
= C e^{-\eps t},
\qquad t\in\bbR,
\end{align}
which for $\ga<1$ contradicts the explicit expression \eqref{chf}.
This contradiction shows that $p(x)$ is not analytic at $x=0=\gd$.  
\end{proof}

\begin{remark}\label{RA} The proof yields also the following.
 For $\ga>1$, and for $\ga=1$ and $\gb\neq0$, 
the density $p(x)$ of $\SSSS\ga\gam\gb\gd$ 
extends to an entire analytic function on $\bbC$.
In the (strictly stable) case $\ga=1$ and $\gb=0$,
the explicit formula \eqref{t1a} shows that that $p(x)$ 
extends to a meromorphic, but not entire, function on $\bbC$.
For $\ga<1$, the restrictions of $p(x)$ to $(-\infty,\gd)$ and
$(\gd,\infty)$ extend to analytic functions $p_+(z)$ and $p_-(z)$
in the slit planes  $\bbC\setminus[\gd,\infty)$  and
$\bbC\setminus[-\infty,\gd)$, respectively, but these two extensions are not
equal.

To verify the claim that $p_+\neq p_-$ when $\ga<1$,
it again suffices to consider the case $\gl=1$ and $\gd=0$,
when the density is $p(x;\ga,\tgam)$.
Note that $p_+(x)$ is obtained by extending \eqref{pgam<} to complex
$x\notin(-\infty,0]$. In particular, it has a jump across the cut that satisfies
\begin{align}
&\lim_{x\to-\infty} |x|^{1+\ga}\bigsqpar{p_+(x+0\ii;\ga,\tgam)-p_+(x-0\ii;\ga\tgam)}
\notag\\&\qquad
=\frac{\gG(\ga+1)}{\pi} \bigpar{e^{-\ii\ga\pi}-e^{\ii\ga\pi}}
\sin\bigpar{\frac{\pi}{2}(\tgam-\ga)}
\notag\\&\qquad
=2\ii\frac{\gG(\ga+1)}{\pi} \sin(\ga\pi)
\sin\bigpar{\frac{\pi}{2}(\ga-\tgam)}
.\end{align}
If $\tgam\in[-\ga,\ga)$, then this limit is non-zero, and thus $p_+$ has a
jump across the cut at least for large  $|x|$. On the other hand, $p_-$ is
analytic across the negative real axis.
If $\tgam=\ga$, we have $p_+(x)=0$, and again we see that $p_+$ and $p_-$
are different.
\end{remark}

\begin{example}\label{EA+}
Consider a positive strictly stable variable; thus 
$\ga<1$, $\gd=0$ and $\tgam=-\ga$ by \refT{T+} and \refE{E+}.
We then have
$p(x;\ga,-\ga)=0$ for $x\le0$ but $p(x;\ga,-\ga)>0$ for $x>0$;
hence, it is in this case 
obvious that the density $p$ is not analytic at 0, as claimed
in \refT{TA}.
(See  \refE{E1/2} for a concrete example.)  
\end{example}

\subsection{Duality}\label{SSdual}

There is a duality due to Zolotarev
between the densities of the distributions of strictly
stable random variables with parameters $\ga$ and $1/\ga$, valid at least for
part of the ranges.

\begin{theorem}[\citet{Z}, \citet{FellerII}]\label{Tdual}
  Let $1\le\ga\le2$ and $|\gth|\le2/\ga-1$, cf.\ \eqref{gthra}.
Define 
$\gth'$ by 
\begin{align}\label{dualgth}
\gth'&=\ga(1+\gth)-1
\in[2\ga-3,1]
.\end{align}
Then, 
\begin{align}\label{dualZ}
  g_C(x;\ga,\gth) = x^{-\ga-1}g_C\bigpar{x^{-\ga};\ga\qw,\gth'},
\qquad x>0.
\end{align}
Equivalently, if $ 0\le A\le B\le\infty$, then
\begin{align}\label{dualZd}
\P\bigsqpar{A<Y_C(\ga,\gth)<B}
=
\frac{1}{\ga}\P\bigsqpar{B^{-\ga}<Y_C(\ga\qw,\gth')<A^{-\ga}}
.\end{align}
Hence,
\begin{align}\label{dualZcond}
  \bigpar{Y_C(\ga,\gth)^{-\ga}\mid Y_C(\ga,\gth)>0}
\eqd\bigpar{Y_C\bigpar{\ga\qw,\gth'}\mid Y_C\bigpar{\ga\qw,\gth'}>0}
.\end{align}

If $1<\ga<2$, we have, equivalently,
\begin{align}\label{dualF}
  p(x;\ga,\tgam)=
x^{-\ga-1}
p\bigpar{x^{-\ga};\ga\qw,\gamx}
\end{align}
with
\begin{align}\label{dualtgam}
 \gamx:=\ga\qw(\tgam+1)-1. 
\end{align}
\end{theorem}

Note that the spectrally negative case $\gth=2/\ga-1$ corresponds to the
positive case $\gth'=1$. (See \refT{T+rho} and \refC{C>0}.)

\begin{proof}
  The relation \eqref{dualZ} is \cite[(2.3.3)]{Z}, and
it is equivalent to \eqref{dualZd} by integration (or, conversely, by
differentiating \eqref{dualZd} with respect to $B$). 
The conditional version \eqref{dualZcond} follows from \eqref{dualZd} (and
is equivalent to it if we also use \refT{T>0}).

Furthermore, for $1<\ga<2$,
\eqref{dualF} is \cite[Lemma XVII.6.2]{FellerII} (with a change of variable),
and it is equivalent to \eqref{dualZ}  by \eqref{gC-p}.

Note also that the cases $\ga=1$ and $\ga=2$ in \eqref{dualZd} follow by
continuity from the case $1<\ga<2$, since the distribution of
$Y_C(\ga,\gth)$ is a continuous function of $(\ga,\gth)$ by \eqref{chfZC}
(with $\gl=1)$.
\end{proof}

The relation \eqref{dualgth} can also be written, using \eqref{rho} and
\eqref{rhor},
\begin{align}\label{dualrho}
  \rho'=\ga\rho\in[\ga-1,1].
\end{align}
(The case $A=0$, $B=\infty$ in \eqref{dualZd} thus is in accordance 
with \refT{T>0}.)

Note that for $1<\ga\le2$, \eqref{dualgth} does not cover the whole range of
$\gth'$ allowed for $Y_C(\ga\qw,\gth',1)$, and similarly for \eqref{dualtgam}.

For $x<0$, we may as usual change signs by \eqref{sd-} and \eqref{gC-}, 
but note that this will change the
relations \eqref{dualgth} and \eqref{dualtgam}.
\refT{Tdual} implies, 
still for $1\le\ga\le2$ 
and $|\gth|\le 2/\ga-1$,
\begin{align}
  g_C(x,\ga,\gth) 
&= g_C(|x|;\ga,-\gth) 
= |x|^{-1-\ga}g_C\bigpar{|x|^{-\ga};\ga\qw,-\gth''}
\notag\\&
= |x|^{-1-\ga}g_C\bigpar{-|x|^{-\ga};\ga\qw,\gth''}
,
\qquad x<0,
\end{align}
with
\begin{align}
  \gth''=1-\ga(1-\gth)=1-\ga+\ga\gth 
\in[-1,3-2\ga].
\end{align}

\subsection{Density at $0$ and $\infty$}\label{SS0}
As said above, the density $g_C(x;\ga,\gth)$ of a strictly stable distribution
$\YY_C\xpar{\ga,\gth}=\YY_C\xpar{\ga,\gth,1}$ 
is always continuous at $x=0$
(although not always analytic there).
Its value is given by a simple formula. 

\begin{theorem}\label{Tg0}
For every $\ga\in(0,2]$ and\/ $\gth$ satisfying \eqref{gthra},
  \begin{align}\label{tg0}
g_C(0;\ga,\gth)    
=\frac{1}{\pi}
\gG\Bigpar{1+\frac{1}{\ga}}\cos\Bigpar{\frac{\pi}2\gth}
=\frac{1}{\pi}
\gG\Bigpar{1+\frac{1}{\ga}}\sin\bigpar{{\pi}\rho}
.  \end{align}
\end{theorem}
\begin{proof}
  The case $\ga\neq1$ is 
\cite[(2.2.11)]{Z}, together with \eqref{gth-gbB} and \eqref{rho}.
 
If $\ga=1$, then 
$\YY_C(1,\gth)=\SSSS1{\cos\frac{\pi\gth}2}0{\sin\frac{\pi\gth}2}$
by \eqref{chfZC} and \eqref{chf} 
(or by \eqref{glC-gl}--\eqref{gth-tgam} and \eqref{joha}--\eqref{johb}),
and \eqref{tg0} follows by \eqref{t1a}.
\end{proof}

As $x\to\infty$, 
we have a corresponding simple asymptotic formula.

\begin{theorem}\label{Tgoo}
For every $\ga\in(0,2]$ and\/ $\gth$ satisfying \eqref{gthra},
  \begin{align}\label{tgoo}
g_C(x;\ga,\gth)    
=\frac{1}{\pi}
\gG\xpar{1+{\ga}}\sin\bigpar{{\pi}\ga\rho}x^{-1-\ga} 
+O\bigpar{x^{-1-2\ga}},
\qquad x\to+\infty
.  \end{align}
\end{theorem}

\begin{proof}
  If $\ga<1$, then \eqref{tgoo} is immediate from \eqref{pgC<}.

If $\ga=1$, then \eqref{tgoo} follows from \eqref{t1a}, noting again that 
$\YY(1,\gth)=\SSSS1{\cos\frac{\pi\gth}2}0{\sin\frac{\pi\gth}2}$
and that $\cos\bigpar{\pi\gth/2}=\sin(\pi\rho)$.

If $\ga>1$, then \eqref{tgoo} follows from \eqref{dualZ}, \eqref{dualrho},
and \eqref{tg0} (applied to $\ga\qw$ and $\rho':=\ga\rho$).
\end{proof}

\section{One-sided moments}\label{Smom}

It is well-known, that for an $\ga$-stable random variable $X$ with
$\ga\neq2$, and $s>0$, we have
\begin{align}
  \E|X|^s<\infty \iff 0<s<\ga.
\end{align}

For strictly stable random variables, these absolute moments can be calculated
explicitly. Moreover, in this case, we can find the moments of 
the positive and negative parts of $X$.
We use the general notation 
$\E\bigsqpar{X;\cE}:=\E\bigsqpar{X\cdot\ett{\cE}}=\int_{\cE} X\dd \P$
for a random variable $X$ and an event $\cE$. 
We then have the following formulas.
Recall that $\gl_C=\gl$ by \eqref{glC-gl}.
\begin{theorem}\label{Tmom}
If\/ $Y=Y_C(\ga,\gth,\gl)$ and $\rho=(1+\gth)/2$, then,
for 
complex $s$ with $ -1 <\Re s <\ga$,
  \begin{align}\label{mom+1}
\E\bigsqpar{Y^s;Y>0}&=
\gl^{s/\ga} \frac{\sin{\pi\rho s}}{\sin\pi s}
\frac{\gG(1-s/\ga)}{\gG(1-s)}
\\& \label{mom+2}
=\frac{1}{\pi}  \gl^{s/\ga} {\sin\xpar{\pi\rho s}}
{\gG(s)}{\gG(1-s/\ga)},
\\&\label{mom+3}
= \gl^{s/\ga} \frac{\gG(s)\gG(1-s/\ga)}{\gG(\rho s)\gG(1-\rho s)}
,
\intertext{and}
\label{mom-1}
\E\bigsqpar{|Y|^s;Y<0}&=
\gl^{s/\ga} \frac{\sin{\pi(1-\rho) s}}{\sin\pi s}
\frac{\gG(1-s/\ga)}{\gG(1-s)}
\\&\label{mom-2}
=\frac{1}{\pi}  \gl^{s/\ga} {\sin\xpar{\pi(1-\rho) s}}
{\gG(s)}{\gG(1-s/\ga)}
\\&\label{mom-3}
= \gl^{s/\ga} \frac{\gG(s)\gG(1-s/\ga)}{\gG((1-\rho) s)\gG(1-(1-\rho)s)}
,
  \end{align}
\end{theorem}

\begin{proof}
\citet[Theorem 2.6.3]{Z} and homogeneity give
\eqref{mom+1}, 
and then \eqref{mom+2}--\eqref{mom+3} follow from
the reflection formula 
$\gG(z)\gG(1-z)=\pi/\sin(\pi z)$.

Since  $-Y\eqd Y_C(\ga,-\gth,\gl)$, and $(1-\gth)/2=1-\rho$,
then \eqref{mom-1}--\eqref{mom-3} follow.
\end{proof}

The absolute moment $\E|Y|^s$ is obtained by summing \eqref{mom+1} and
\eqref{mom-1}.

Note that the special case $s=0$ (when the formulas are interpreted in the
obvious ways, taking limits) yields
$\P\sqpar{Y>0}=\rho$ and $\P\sqpar{Y<0}=1-\rho$, as stated in \refT{T>0}.
Consequently, we obtain the conditional moments
$\E\bigsqpar{Y^s\mid Y>0}$
and $\E\bigsqpar{|Y|^s\mid Y<0}$
by dividing \eqref{mom+1}--\eqref{mom+3} and \eqref{mom-1}--\eqref{mom-3} 
by $\rho$ and $1-\rho$, respectively.

When $\Re s>0$, we can also interpret \eqref{mom+1}--\eqref{mom-3} as the
moments of
$Y_+:=\max\set{Y,0}$
and
$Y_-:=\max\set{-Y,0}$. 

\begin{remark}
  If $Y$ has density $p(x)$, then
$\E\bigsqpar{Y^s; Y>0}=\intoo x^sp(x)\dd x$
and $\E\bigsqpar{|Y|^s; Y<0}=\intoo x^{s}p(-x)\dd x$.
Hence, \eqref{mom+1}--\eqref{mom-3} can be regarded as formulas for the
Mellin transforms of $p$ restricted to the positive and negative half-axes.
\end{remark}

\begin{remark}\label{RMellin}
The range $-1<\Re s<\ga$ in \refT{Tmom} is in most cases optimal.
In fact, it follows from \eqref{mom+2} that
$\E\bigsqpar{Y^s; Y>0}$ has a pole as $s=-1$ unless $\sin(-\pi\rho)=0$,
\ie, $\rho=0$ or $\rho=1$; in both cases $\ga<1$ by \eqref{rhor}.
Similarly, \eqref{mom+2}
shows that $s=\ga$ is a pole unless
$\sin(\pi\rho\ga)=0$, i.e., $\rho=0$ (and then $\ga<1$), or $\rho=1/\ga$
(and then $\ga>1$).
These exceptional cases are treated in the examples below.
In all other cases, we thus have poles at $-1$ and $\ga$, and, consequently,
$\E\bigsqpar{Y^s; Y>0}=\infty$ for $s\le -1$  or $s\ge \ga$.
\end{remark}

\begin{example}\label{EM0}
  If $\ga<1$ and $\rho=0$, then $Y<0$ a.s.\ by \refT{T>0}, and thus,
  trivially,
$\E\bigsqpar{Y^s; Y>0}=0$ for all $s$, which agrees with \eqref{mom+1}.
\end{example}

\begin{example}\label{EM1}
If $\ga<1$ and $\rho=1$, then $Y>0$ a.s.\   by \refT{T>0}, i.e., $Y$ is a
positive strictly stable random variable as in \refE{E+}. Hence its
infinitely differentiable density
$p(x)$ vanishes on $(-\infty,0)$, and thus has all derivates $=0$ at 0,
whence $p(x)=O(x^N)$ as $x\to0$ for any $N>0$.
It follows that 
$\E\bigsqpar{Y^s}$ is finite for all $s<0$, and thus analytic in $\Re s<\ga$.
By \refR{RMellin}, there is a pole at $\ga$. By \eqref{mom+1} and analytic
continuation,
\begin{align}\label{em1}
  \E \bigsqpar{Y^s} = 
\gl^{s/\ga} 
\frac{\gG(1-s/\ga)}{\gG(1-s)},
\qquad \Re s<\ga.
\end{align}
\end{example}

\begin{example}\label{EM1/ga}
If $1<\ga<2$ and $\rho=1/\ga$, then 
$Y$ is spectrally negative by \refT{T+rho}.
Hence, by \refT{Texp} and a change of signs, 
the \mgf{} $\E e^{tY}<\infty$ for every $t\ge0$,
and it follows that 
$\E\bigsqpar{Y^s; Y>0}<\infty$ for all $s>0$.
Hence,
\eqref{mom+3} and analytic continuation yield
\begin{align}\label{em1/ga}
  \E\bigsqpar{Y^s;Y>0}&
= \gl^{s/\ga} \frac{\gG(s)}{\gG(s/\ga)}
,\qquad \Re s >-1.
\end{align}
This holds for $\ga=2$ too, when \refT{Texp} as stated does not apply,
because then $Y$ is normal and the \mgf{} is finite everywhere.
\end{example}

\begin{example}\label{EM1-1/ga}
If $1<\ga<2$ and $\rho=1-1/\ga$, then 
$Y$ is spectrally positive by \refT{T+rho}, 
and $-Y$ is as in \refE{EM1/ga}.
Hence, 
$\E\bigsqpar{|Y|^s;Y<0}$ is finite for $\Re s>-1$,
while 
$\E\bigsqpar{Y^s;Y>0}$ has a pole at $\ga$.
\end{example}

\section{Some  examples}\label{Sex}

\begin{example}[$\ga=2$]\label{E2}
The case $\ga=2$ is simple, and also exceptional in several ways.
By \eqref{chf}, the distribution $\SSSS2\gam\gb\gd$ has \chf{}
\begin{align}\label{e2a}
  \gf(t)=e^{\ii\gd t-\gam^2 t^2},
\end{align}
and thus a 2-stable distribution is nornal:
$\SSSS2\gam\gb\gd=N(\gd,2\gam^2)$.
As said in \refT{Tchf}, this distribution does not depend on $\gb$, and we take
$\gb=0$.

Conversely, we see that a normal distribution $N(\mu,\gss)$ is 2-stable,
with, by \eqref{e2a} and \eqref{chfZA}--\eqref{chfZB},
\begin{align}
&\gam=\frac{1}{\sqrt2}\gs,\quad
\gd=\mu,\quad
\gl_A=\gl_B=\gl_M=\frac12\gss,\quad
\gam_A=\gam_B=\gam_M=\frac{2\mu}{\gss}
.\end{align}
The distribution is strictly stable if and only if its mean $\mu=0$
(see \refR{Rstrict}), and
then we further have,
by \eqref{jepp1}, \eqref{joh}, \eqref{chfZC},  and \eqref{rho},
\begin{align}
\kk=
\gl=
\gl_C=\frac12\gss,\quad
\tau=
\tgam=
\gth=0,\quad 
\rho=\frac12
.\end{align}
(Cf.\ \eqref{tgamr}, \eqref{gthra}, \eqref{rhor}.)

In particular, $\SSSS2100=\YY_C(\frac12,0)$ has the density
\begin{align}\label{e2p}
  p(x;2,0)=g_C(x;2,0)=\frac{1}{2\sqrt\pi}e^{-x^2/4}.
\end{align}

The normal distribution has \Levy{} measure $\gL=0$, and the canonical measure
$M$ is a point mass at $\set0$, 
with $M\set0=\gss$;
see \eqref{tlevy} and \eqref{tcan}.
\end{example}

\begin{example}[$\ga=1$]\label{E1}
  The Cauchy distribution has density
  \begin{align}
    f(x)=\frac{1}{\pi(1+x^2)},
\qquad -\infty<x<\infty,
  \end{align}
and \chf{}
\begin{align}\label{Cauchy-chf}
  \gf(t)=e^{-|t|}
,\qquad -\infty<t<\infty
.\end{align}
The Cauchy distribution is thus strictly 1-stable. 
More precisely, by \eqref{chf}, it is $\SSSS1100=\SSSx1(0)$;
see also \refT{T1}\ref{T1=}.
We thus have, using also \eqref{jepp1} or \eqref{jepp1a}, 
\eqref{joh}--\eqref{johb}, 
\eqref{gbAM}--\eqref{glAM}, 
\eqref{gbB1}--\eqref{glB1},
\eqref{glC-gl}--\eqref{gth-tgam},
and \eqref{rho},
\begin{align}\label{e1c}
&\gam=1,\quad
\gb=\gd=0,\quad
\kk=1,\quad
\tau=0,\quad
\gl=1,\quad
\tgam=0,\quad
\notag\\&
  \gb_A=\gb_B=\gb_M=0,\quad
\gam_A=\gam_B=\gam_M=0,\quad
\gl_A=\gl_M=\gl_C=1,\quad
\gl_B=2/\pi,\quad
\notag\\&
\gth=0,\quad 
\rho=\frac12
.\end{align}

By \refT{Tchf}, the strictly 1-stable distributions are $\SSSS1\gam0\gd$,
and by \eqref{lin},
\begin{align}
  \XXXX1\gam0\gd \eqd \gam\XXXx1(0)+\gd.
\end{align}
In other words, the strictly 1-stable  distributions are precisely the
linear transformations of the Cauchy distribution.

If we normalize to $\gam=1$, we have,
generalizing \eqref{e1c},
that the strictly stable distribution $\SSSS110\gd$ has,
by
\refR{R11}, \eqref{jepp1}, \eqref{joh}, \eqref{johb}, 
\eqref{gbAM}--\eqref{glAM}, 
\eqref{gbB1}--\eqref{gbbb}, 
\eqref{glC-gl}--\eqref{gth-tgam},
and \eqref{rho},
\begin{align}
&
\kk=
\gam=1,\quad
\tau=\gd,\quad
\gl=\gl_C=\sqrt{1+\gd^2},\quad
\tgam=- \frac{2}{\pi}\arctan \gd,\quad
\notag\\&
\gb= \gb_A=\gb_B=\gb_M=0,\quad
\gam_A=\gam_M=\gd,\quad
\gam_B=\frac{\pi\gd}2,\quad
\gl_A=\gl_M=1,\quad
\notag\\&
\gl_B=\frac{2}{\pi},\quad
\gth= \frac{2}{\pi}\arctan \gd,\quad
\rho=\frac12 + \frac{1}{\pi}\arctan \gd
.\end{align}
\end{example}

\begin{example}[$\ga=\xfrac12$]  \label{E1/2}
The positive $\frac12$-stable distribution 
is closely connected to the normal distribution and Brownian motion.

One way to see this is to consider a standard Brownian motion $B_t$, $0\le
t<\infty$, and for $a\ge0$ let $T_a$ be the hitting time
$T_a:=\inf\set{t\ge0:B_t\ge a}$.
Then, by Brownian scaling,
$T_a\eqd a^2 T_1$, and by the strong Markov property,
$T_{a+b}-T_a\eqd T_b$, for $a,b\ge0$.
Hence, if $X=T_1$, then 
\begin{align}
  S_n:=\sumin X_i \eqd T_n \eqd n^2 X,
\end{align}
which shows that $X=T_1$ is strictly $\frac12$-stable. Obviously, $T_1>0$.
More generally, $(T_a)_{a\ge0}$ is an increasing stable process (i.e., a
L\'evy process with stable increments,
see \refR{Rproc} and \eg{} \cite{Bertoin}).

A simple calculation using the martingale $e^{\sqrt{2t} B_x- tx}$,
$x\ge0$,
see \eg{}
\cite[Proposition II.3.7]{RY},
gives the Laplace transform
\begin{align}
  \E e^{-t T_1}=e^{-\sqrt{2t}},
\qquad t\ge0.
\end{align}
Hence,
by \refE{E+} (with $\gl=\sqrt2$),
$T_1\sim \SSSS{1/2}110$.
Using also \refT{T+},
\eqref{jepp1a}, \eqref{johc},
\eqref{gbAM}--\eqref{glB}, 
\eqref{glC-gl}--\eqref{gth-tgam},
and \eqref{rho},
\begin{align}
&\gam=1,\quad
\gb=1,\quad
\gd=0,\quad
\kk=1,\quad
\tau=1,\quad
\gl=\gl_C=\sqrt2,\quad
\tgam=-\frac12,\quad
\notag\\&
  \gb_A=\gb_B=\gb_M=1,\quad
\gam_A=\gam_B=0,\quad
\gam_M=1,\quad
\gl_A=\gl_M
=1,\quad
\gl_B=\sqrt2,\quad
\notag\\&
\gth=
\rho=1
.\end{align}

More generally, for any $a>0$,
$T_a\sim \SSSS{1/2}{a^2}10 = \YY_C(1/2,1,a\sqrt2)$.

Moreover, using the reflection principle
\cite[Proposition III.3.7]{RY}, for any $x>0$,
\begin{align}
  \P(T_1\le x)&
=\P\bigpar{\sup_{0\le t\le x}B_t\ge1}
=2\P\bigpar{B_x\ge1}
=\P\bigpar{|B_x|\ge1}
\notag\\&
=\P\bigpar{x\qq|B_1|\ge1}
=\P\bigpar{|B_1|^2\ge 1/x}
=\P\bigpar{|B_1|^{-2}\le x}.
\end{align}
Hence,
\begin{align}\label{tba}
  T_1\eqd B_1^{-2},
\qquad\text{where }
B_1\sim N(0,1).
\end{align}
In other words, if $Z\sim N(0,1)$, then
$Z\qww\sim \SSSS{1/2}110 =  \YY_C(1/2,1,\sqrt2)$.

From \eqref{tba}, $T_1$ has the density
\begin{align}\label{tbb}
  f_{T_1}(x)=\frac{1}{\sqrt{2\pi x^3}} e^{-1/(2x)},
\qquad x>0.
\end{align}
This follows also from \eqref{dualZ}.
Hence, if $X\sim\SSSS{\xfrac12}\gam10
=  \YY_C(1/2,1,\sqrt{2\gam})$, then $X\eqd \gam T_1$ has density
\begin{align}\label{tbc}
  f_X(x)
=\frac{\gam\qq}{\sqrt{2\pi x^3}} e^{-\gam/(2x)},
\qquad x>0.
\end{align}
Taking $\gam=1/2$, we find
\begin{align}\label{tbg}
g_C(x;1/2,1)
=\frac{1}{2\sqrt{\pi x^3}} e^{-1/(4x)},
\qquad x>0,
\end{align}
which agrees with \eqref{dualZ} and \eqref{e2p}.
\end{example}

\begin{example}[$\ga=\xfrac32$]\label{E3/2}
  
\citet{map-Airy} define a $\frac32$-stable distribution, 
by them called \emph{the Airy distribution of map type}; 
it has a density $\cA(x)$ given by
\cite[(B.2)]{map-Airy}
\begin{align}\label{e3/2a}
  \cA(x)=\frac{1}{2\pi\ii}\int_{-\infty\ii}^{\infty\ii} e^{-xt+t^{3/2}/3}\dd  t
=\frac{1}{2\pi}\intoooo e^{-\ii xt+(\ii t)^{3/2}/3}\dd t,
\end{align}
which can be recognized as the inversion formula for a distribution with
\chf{}
\begin{align}\label{e3/2chf}
  \gf(t) = e^{(\ii t)^{3/2}/3}
=\exp\Bigpar{-\frac{1}3 e^{-\ii\pi\sgn(t)/4} |t|^{3/2}}
=\exp\Bigpar{-\frac{1}{3\sqrt2}(1-\ii\sgn(t)) |t|^{3/2}}.
\end{align}
This is thus (as noted in \cite{map-Airy}) a $\frac32$-stable distribution;
more precisely, 
by comparing with  \eqref{jepp1} and \eqref{chfZC}, we see that this is the
strictly stable distribution with,
using also \eqref{rho},
\begin{align}\label{e3/2b}
\ga=3/2, 
\quad \kk=\tau=\frac{1}{3\sqrt2}=\frac{1}{\sqrt{18}},
\quad \gl_C=\frac{1}3,
\quad  \gth=\frac13,
\quad &\rho=\frac23 
.\end{align}
We find also, using
\eqref{jepp1b}, \eqref{joh}
or \eqref{johd},
\eqref{gbAM}--\eqref{glB}, \eqref{glC-gl}, \eqref{glC-glB}, 
\begin{align}\label{e3/2c}
&\gam= 2^{-1/3}3^{-2/3}=18^{-1/3},\quad
\gb=-1,\quad
\gd=0,\quad
\gl=\frac13,\quad
\tgam=-\frac12,
\notag\\&
\gb_A=\gb_B=\gb_M=-1,\quad
\gam_A=\gam_B=0,\quad
\gam_M=1,\quad
\gl_A=\gl_M=\frac{1}{3\sqrt2},\quad
\notag\\&\gl_B=\gl_C=\frac{1}3
.\end{align}
The distribution is thus spectrally negative.
If $X$ has this distribution, then by \eqref{e+2} applied to $-X$,
\begin{align}\label{e3/2mgf}
  \E e^{tX} = \exp\bigpar{\tfrac13 t^{3/2}},
\qquad\Re t\ge0.
\end{align}

It is shown in \cite{map-Airy} that the density \eqref{e3/2a}
also can be expressed as
\begin{align}\label{e3/2Ai}
\cA(x):=2 e^{-2x^3/3}\bigpar{x\Ai(x^2)-\Ai'(x^2)}  ,
\qquad -\infty<x<\infty,
\end{align}
where $\Ai(x)$ is the Airy function 
\cite[Chapter 9]{DLMF}.

This distribution is of the type in \refE{EM1/ga},
and \eqref{em1/ga} yields
\begin{align}\label{mapmom+}
\intoo x^s \cA(x)\dd x &
= 3^{-2s/3} \frac{\gG(s)}{\gG(2s/3)}
,\qquad \Re s >-1.
\end{align}
For the negative side, we have by \eqref{mom-2} and the reflection formula
for the Gamma function,
\begin{align}\label{mapmom-}
\int_{-\infty}^0 |x|^s \cA(x)\dd x &
= \frac{1}{\pi}3^{-2s/3}\sin\frac{\pi s}{3}\, {\gG(s)}{\gG(1-2s/3)}
\notag\\&
= 3^{-2s/3}\frac{\sin\frac{\pi s}{3}}{\sin\frac{2\pi s}{3}} 
\frac{\gG(s)}{\gG(2s/3)}
\notag\\&
= 2^{-1}3^{-2s/3}\frac{1}{\cos\frac{\pi s}{3}} \frac{\gG(s)}{\gG(2s/3)}
,\qquad -1<\Re s <3/2.
\end{align}
The formulas \eqref{mapmom+} and \eqref{mapmom-}  
are equivalent to  \cite[(B.5)--(B.6)]{map-Airy}.

By \eqref{e3/2c} and \eqref{sd1}, the density
\begin{align}
\cA(x)
=g_C\bigpar{x;3/2,1/3,1/3} 
= 3^{2/3}p\bigpar{3^{2/3}x;3/2,-1/2}
\end{align}
and thus, by  \eqref{gC-p} and \eqref{e3/2Ai}, 
\begin{align}\label{gC3/2}
&  
 g_C\bigpar{x;3/2,1/3} 
=p\bigpar{x;3/2,-1/2} 
= 3^{-2/3}\cA\bigpar{3^{-2/3}x}
\notag\\&\qquad
=
2\cdot 3^{-2/3} e^{-2 x^3/27} \Bigpar{3^{-2/3}x\Ai\bigpar{3^{-4/3}x^2}
-\Ai'\bigpar{3^{-4/3}x^2}}.
\end{align}

An alternative formula using the Whittaker function $W_{\kk,\mu}$
\cite[\S13.14]{DLMF} is \cite[(2.8.34) with a typo]{Z}:
\begin{align}\label{gC3/2W}
 g_C\bigpar{x;3/2,1/3} 
 =\frac{\sqrt3}{\sqrt\pi}x\qw e^{-2x^3/27} W_{1/2,1/6}\Bigpar{\frac{4x^3}{27}},
\qquad x>0
.\end{align}
For the negative side we have, by \eqref{gC-} and \cite[(2.8.35)]{Z},
\begin{align}\label{gC3/2W-}
 g_C\bigpar{x;3/2,1/3} &
=g_C\bigpar{|x|;3/2,-1/3} 
\notag\\&
 =\frac{1}{2\sqrt{3\pi}}|x|\qw e^{2|x|^3/27} W_{-1/2,1/6}\Bigpar{\frac{4|x|^3}{27}},
\qquad x<0
.\end{align}
Of course, the corresponding spectrally positive distribution
$\YY_C(3/2,-1/3)$ has density $g_C(-x;3/2,1/3)$ obtained by switching
\eqref{gC3/2W} and \eqref{gC3/2W-}.
\end{example}

\begin{example}[$\ga=\xfrac23$]  \label{E2/3}
The positive strictly $\frac23$-stable distribution with Laplace transform
\begin{align}\label{e2/3L}
  \E e^{-tX} = \exp\bigpar{- t^{2/3}},
\qquad\Re t\ge0,
\end{align}
is $\SSSS{2/3}{2^{-3/2}}{1}{0}=\YY_C(2/3,1)=\YY_C(2/3,1,1)$ by \refEs{E+} and
\ref{E+Z}. 
By \eqref{gC-p} and \eqref{dualZ} (with $\ga=3/2$ and $\gth=1/3$), 
its density function is
\begin{align}\label{e2/3g}
 g_C(x;2/3,1)
=p\bigpar{x;2/3,-2/3} 
=x^{-5/3}g_C\bigpar{x^{-2/3};  3/2,1/3},
\qquad x>0.
\end{align}
By \eqref{gC3/2}, this yields the density,
for $x>0$,
\begin{align}\label{gC2/3}
&g_C(x;2/3,1)  =
6 e^{-\frac{2}{27x^2}} \Bigpar{(3x)^{-7/3}\Ai\bigpar{(3x)^{-4/3}}
- (3x)^{-5/3}\Ai'\bigpar{(3x)^{-4/3}}}.
\end{align}
Similarly, \eqref{e2/3g} and \eqref{gC3/2W} yield \cite[(2.8.33) with typo]{Z}
\begin{align}\label{gC2/3W}
 g_C\bigpar{x;2/3,1} 
 =\frac{\sqrt3}{\sqrt\pi}x\qw e^{-\frac{2}{27 x^2}} W_{1/2,1/6}\Bigpar{\frac{4}{27x^2}},
\qquad x>0
.\end{align}
\end{example}

\begin{example}[$\ga=\xfrac23$]  \label{E2/3=}
  The symmetric $\frac23$-stable distribution with characteristic function
\begin{align}\label{e2/3=}
  \E e^{\ii tX} = \exp\bigpar{- |t|^{2/3}},
\qquad -\infty < t <\infty,
\end{align}
is $\SSSS{2/3}{1}{0}{0}=\YY_C(2/3,0)=\YY_C(2/3,0,1)$ by 
\eqref{chf} and \eqref{chfZC}.

By symmetry, \eqref{gC-p} and \eqref{dualZ} (with $\ga=3/2$ and $\gth=-1/3$), 
the density function is
\begin{align}\label{e2/3g=}
 g_C(x;2/3,0)&
=p\bigpar{x;2/3,0} 
=|x|^{-5/3}g_C\bigpar{|x|^{-2/3};  3/2,-1/3},
\end{align}
which by \eqref{gC3/2W-}  yields              
\cite[(2.8.32)]{Z}
\begin{align}\label{gC2/3W=}
 g_C(x;2/3,0)&
 =\frac{1}{2\sqrt{3\pi}}|x|\qw e^{\frac{2}{27x^2}} W_{-1/2,1/6}\Bigpar{\frac{4}{27x^2}},
\qquad x\neq0
.\end{align}
\end{example}

\begin{example}[$\ga=\xfrac13$]  \label{E1/3}
The positive strictly $\frac13$-stable distribution with Laplace transform
\begin{align}\label{e1/3L}
  \E e^{-tX} = \exp\bigpar{- t^{1/3}},
\qquad\Re t\ge0,
\end{align}
is $\SSSS{1/3}{(3/4)^{3/2}}{1}{0}=\YY_C(1/3,1)=\YY_C(1/3,1,1)$ by \refEs{E+} and
\ref{E+Z}. 

The density function is, by \cite[(2.8.31)]{Z} and \cite[(9.6.1)]{DLMF},
\begin{align}\label{E1/3g}
g_C\bigpar{x;1/3,1}
=p\bigpar{x;1/3,-1/3}
  = 3^{-1/3} x^{-4/3} \Ai \bigpar{(3x)^{-1/3} },
\qquad x>0,
\end{align}
where $\Ai(x)$ again is the Airy function. 
Equivalently, $3\Ai(x)$, $x>0$, is the density of the random variable
$(3Y_C(1/3,1))^{-1/3}$. 
(The distribution of this variable, apart from the factor $3^{-1/3}$,
is known as a Mittag--Leffler distribution).

The moment formula \eqref{em1} with $\ga=1/3$ is by \eqref{E1/3g}
and a change of variables
equivalent to the integral formula \cite[(9.10.17)]{DLMF}
\begin{align}
  \intoo x^{a-1}\Ai(x)\dd x = 3^{-(\ga+2)/3}\frac{\gG(a)}{\gG((\ga+2)/3)},
\qquad \Re a >0.
\end{align}
\end{example}

\section{Domains of attraction}

\begin{definition}
  A random variable $X$ belongs to the \emph{domain of attraction} of a
  stable distribution $\cL$
if there exist constants $a_n>0$ and $b_n$ such that 
\begin{equation}\label{da}
  \frac{S_n-b_n}{a_n}\dto \cL
\end{equation}
as \ntoo, where $S_n\=\sumin X_i$ is a sum of $n$ \iid{} copies of $X$.
\end{definition}

We will in the sequel always use the notation $S_n$ in the sense above
(as we already have done in \refS{Sstab}).
All unspecified limits are as \ntoo.

\begin{theorem}\label{TD1}
Let $0<\ga\le 2$.
  A (non-degenerate)
random variable $X$ belongs to the domain of attraction of an
  $\ga$-stable distribution if and only if the following 
two conditions hold:
\begin{romenumerate}[-10pt]
\item 
the truncated moment function
  \begin{equation}\label{mu}
	\mu(x)\=\E \bigpar{X^2\ett{|X|\le x}}
  \end{equation}
varies regularly with exponent $2-\ga$ as $\xtoo$, \ie,
\begin{equation}\label{td1a}
  \mu(x)\sim x^{2-\ga} L_1(x),
\end{equation}
where $L_1(x)$ varies slowly;
\item 
either $\ga=2$, or 
the tails of $X$ are balanced:
\begin{equation}\label{td1b}
\frac{\P(X>x)}{\P(|X|>x)}\to p_+,\qquad \xtoo,
\end{equation}
for some $p_+\in\oi$.
\end{romenumerate}
\end{theorem}

\begin{proof}
  \citet[Theorem XVII.5.2]{FellerII}.
\end{proof}

For the case $\ga<2$, the following version is often more convenient.

\begin{theorem}\label{TD2}
Let $0<\ga< 2$.
  A 
random variable $X$ belongs to the domain of attraction of an
  $\ga$-stable distribution if and only if the following 
two conditions hold:
\begin{romenumerate}[-10pt]
\item 
the tail probability $\P(|X|>x)$
varies regularly with exponent $-\ga$ as $\xtoo$, \ie,
\begin{equation}\label{td2a}
\P(|X|>x)\sim x^{-\ga} L_2(x),
\end{equation}
where $L_2(x)$ varies slowly;
\item 
the tails of $X$ are balanced:
\begin{equation}\label{td2b}
\frac{\P(X>x)}{\P(|X|>x)}\to p_+,\qquad \xtoo,
\end{equation}
for some $p_+\in\oi$.
\end{romenumerate}
\end{theorem}

\begin{proof}
  \citet[Corollary XVII.5.2]{FellerII}.
\end{proof}

We turn to identifying the stable limit distributions in Theorems
\ref{TD1}--\ref{TD2} explicitly.

\subsection{The case $\ga<2$}

If the conditions of \refT{TD1} or \ref{TD2} hold for some $\ga<2$, then the
conditions of the other hold too, and we have,
by   \cite[(5.16)]{FellerII},
\begin{equation}
  L_2(x)\sim \frac{2-\ga}{\ga} L_1(x), \qquad \xtoo.
\end{equation}
Furthermore, by   
\cite[(5.6)]{FellerII}, 
with $a_n,b_n$ as in \eqref{da}
and $M$ and $\gL$ the canonical measure and \Lm{} of the limit distribution
$\cL$, 
\begin{equation}\label{lm1}
  n\P(X>a_nx) \to \gL(x,\infty)
=\int_x^\infty y\qww \dd M(y), \qquad x>0, 
\end{equation}
and, by symmetry,
\begin{equation}\label{lm2}
  n\P(X<-a_nx) \to \gL(-\infty,-x)
=\int_{-\infty}^x y\qww \dd M(y), \qquad x>0. 
\end{equation}
In particular,
\begin{equation}
  n\P(|X|>a_n) \to \gL\set{y:|y|>1}\in(0,\infty);
\end{equation}
conversely, we may in \eqref{da} choose any sequence $(a_n)$ such that 
$  n\P(|X|>a_n)$ converges to a positive, finite limit.
(Any two such sequences $(a_n)$ and $(a_n')$ must satisfy $a_n/a_n'\to c$
for some $c\in\oooy$, as a consequence of \eqref{td2a}.)

If
\begin{equation}\label{bro}
  n\P(|X|>a_n) \to C>0
\end{equation}
and \eqref{td2a}--\eqref{td2b} hold, then \eqref{lm1}--\eqref{lm2} hold with 
$\gL(x,\infty)=p_+Cx^{-\ga}$
and $\gL(-\infty,-x)=p_-Cx^{-\ga}$, where $p_-\=1-p_+$. Hence, 
\eqref{tsM}--\eqref{tsL} hold with
\begin{align}
  c_+=p_+C\ga,&&&
  c_-=p_-C\ga.
\end{align}
Consequently, the limit distribution is given by \eqref{chf} where, by
\eqref{chf3a}--\eqref{chf3b},
\begin{align}
  \gam
&=\bigpar{C\ga\bigpar{-\Gamma(-\ga)\cos\tpiaq}}\xga
=\bigpar{C\Gamma(1-\ga)\cos\tpiaq}\xga,
\label{sw1}
\\
\gb&=p_+-p_-.  
\label{sw2}
\end{align}
For $\ga=1$ we interpret \eqref{sw1} by continuity as
\begin{equation}
    \gam
=C\tfrac\pi2 .
\label{sw1.1}
\end{equation}

\begin{theorem}\label{Tlim}
  Let $0<\ga<2$. Suppose that \eqref{td2a}--\eqref{td2b} hold and that $a_n$
  are chosen such that \eqref{bro} holds, for some $C$. 
Let $\gam$ and $\gb$ be defined by \eqref{sw1}--\eqref{sw2}.
  \begin{romenumerate}[-10pt]
\item 
If $0<\ga<1$, then
\begin{equation}
  \frac{S_n}{a_n}\dto \SSSx\ga(\gam,\gb,0).
\end{equation}
\item 
If $1<\ga<2$, then
\begin{equation}
  \frac{S_n-n\E X}{a_n}\dto \SSSx\ga(\gam,\gb,0).
\end{equation}
\item 
If $\ga=1$, then
\begin{equation}
  \frac{S_n-nb_n}{a_n}\dto \SSSx1(\gam,\gb,0),
\end{equation}
where 
$\gam$ is given by \eqref{sw1.1} and
\begin{equation}
  b_n\=a_n\E \sin(X/a_n).
\end{equation}
  \end{romenumerate}
\end{theorem}
\begin{proof}
  \citet[Theorem XVII.5.3]{FellerII} together with the calculations above.
\end{proof}

\begin{example}\label{Epower}
  Suppose that $0<\ga<2$ and that $X$ is a random variable such that, as \xtoo,
  \begin{equation}\label{epower}
	\P(X>x)\sim Cx^{-\ga},
  \end{equation}
with $C>0$,
and $\P(X<-x)=o(x^{-\ga})$.
Then \eqref{td2a}--\eqref{td2b} hold with $L_2(x)\=C$ and $p_+=1$, and thus
$p_-\=1-p_+=0$. 
We take $a_n\=n\xga$; then \eqref{bro} holds, 
and thus \eqref{tsM}--\eqref{tsL} hold with
\begin{align}
  c_+=C\ga,&&&
  c_-=0;
\end{align}
hence, 
\eqref{sw1}--\eqref{sw2} yield
\begin{equation}
    \gam
=\bigpar{C\Gamma(1-\ga)\cos\tpiaq}\xga,
\label{sw3}
\end{equation}
and $\gb=1$.
Consequently, \refT{Tlim} yields the following.
  \begin{romenumerate}[-10pt]
\item 
If $0<\ga<1$, then
\begin{equation}
  \frac{S_n}{n\xga}\dto \SSSx\ga(\gam,1,0).
\end{equation}
The limit variable $Y$
is positive and has by \refT{Texp} and \eqref{sw3}
the Laplace transform
\begin{equation}\label{sw4}
\E e^{-tY}  =\exp\bigpar{-C\Gamma(1-\ga)t^\ga},
\qquad \Re t\ge0.
\end{equation}
\item 
If $1<\ga<2$, then
\begin{equation}
  \frac{S_n-n\E X}{n\xga}\dto \SSSS\ga\gam10.
\end{equation}
The limit variable $Y$
has by \refT{Texp} and \eqref{sw3}
the finite Laplace transform
\begin{equation}\label{sw5}
\E e^{-tY}  =\exp\bigpar{C|\Gamma(1-\ga)|t^\ga},
\qquad \Re t\ge0.
\end{equation}
By \eqref{sd2} and \eqref{sw3}, the density function $f_Y$ of the limit
variable satisfies
\begin{equation}
  f(0)=C\xgaw |\Gamma(1-\ga)|\xgaw|\Gamma(-1/\ga)|\qw.
\end{equation}

\item 
If $\ga=1$, then
\begin{equation}\label{emm1}
  \frac{S_n-nb_n}{n}
=
  \frac{S_n}{n}-b_n
\dto \SSSx1(\gam,1,0),
\end{equation}
where, by \eqref{sw1.1}, 
$\gam=C\pi/2$ and
\begin{equation}\label{emm2}
  b_n\=n\E \sin(X/n).
\end{equation}
We return to the evaluation of $b_n$ in \refS{S1}.
  \end{romenumerate}
\end{example}

\begin{example}
  Suppose that $0<\ga<2$ and that $X\ge0$ is an integer-valued random
  variable such that, as \ntoo, 
  \begin{equation}\label{epower0}
	\P(X=n)\sim c n^{-\ga-1}.
  \end{equation}
Then \eqref{epower} holds with 
\begin{equation}
C=c/\ga  
\end{equation}
and the results of \refE{Epower} hold, with this $C$. In particular,
\eqref{sw3} yields
\begin{equation}
  \gam^\ga = -c\,\Gamma(-\ga)\cos\tpiaq,
\end{equation}
and both \eqref{sw4} and \eqref{sw5} can be written
\begin{equation}\label{sjw6}
\E e^{-tY}  =\exp\bigpar{c\,\Gamma(-\ga)t^\ga},
\qquad \Re t\ge0;
\end{equation}
note that $\Gamma(-\ga)<0$ for $0<\ga<1$
but $\Gamma(-\ga)>0$ for $1<\ga<2$.

Taking $t$ imaginary in \eqref{sjw6}, we find the \chf{}
\begin{equation}
\E e^{\ii tY}  =\exp\bigpar{c\,\Gamma(-\ga)(-\ii t)^\ga}
=\exp\bigpar{c\,\Gamma(-\ga)e^{-\ii\sgn(t) \pi\ga/2} |t|^\ga},
\qquad t\in\bbR.
\end{equation}
\end{example}

\subsection{The special case $\ga=1$}\label{S1}

Suppose that, as \xtoo, 
\begin{equation}\label{em}
	\P(X>x)\sim Cx^{-1}  
\end{equation}
and $\P(X<-x)=o(x^{-1})$, with $C>0$.
Then \refE{Epower} applies, and \eqref{emm1}--\eqref{emm2} hold. We
calculate the normalising quantity $b_n$ in \eqref{emm1} for some examples.

\begin{example}\label{E1u}
  Let $X\=1/U$, where $U\sim \U(0,1)$ has a uniform distribution.
Then 
$\P(X>x)=x\qw$ for $x\ge 1$ so \eqref{em} holds with $C=1$ and
\eqref{sw1.1} yields $\gam=\pi/2$.  
Furthermore,
$X$ has a Pareto distribution with the density
\begin{equation}
  f(x)=
  \begin{cases}
	x\qww,& x>1,\\
0,& x\le1.
  \end{cases}
\end{equation}
Consequently, by \eqref{emm2},
\begin{equation*}
  \begin{split}
b_n&=  n\sin(X/n)
=n\int_1^\infty \sin(x/n)x\qww\dd x
=\int_{1/n}^\infty \sin(y)y\qww\dd y
\\&
=\log n+\int_{1/n}^1 \frac{\sin y-y}{y^2}\dd y	
 +\int_{1}^\infty \frac{\sin y}{y^2}\dd y	
\\&
=\log n+\int_{0}^\infty \frac{\sin y-y\ett{y<1}}{y^2}\dd y+o(1)	
=\log n+1-\ggam+o(1),
  \end{split}
\end{equation*}
where $\ggam$ is Euler's gamma.
(For the standard evaluation of the last integral, see \eg{} \cite{SJN11}.)
Hence, \eqref{emm1} yields
\begin{equation}
  \frac{S_n}{n}-\bigpar{\log n+1-\ggam}
\dto \SSSx1(\pi/2,1,0).
\end{equation}
or
\begin{equation}
  \frac{S_n}{n}-\log n
\dto \SSSx1(\pi/2,1,1-\ggam).
\end{equation}
\end{example}

\begin{example}\label{E1exp}
  Let $X\=1/Y$, where $Y\sim \Exp(1)$ has an exponential distribution.
Then 
$\P(X>x)=1-\exp(-1/x)\sim x\qw$ as \xtoo{} so $C=1$ and \eqref{sw1.1} yields
$\gam=\pi/2$.  
In this case we do not calculate $b_n$ directly from \eqref{emm2}.
Instead we define $U\=1-e^{-Y}$ and $X'\=1/U$ and note that $U$ has a
uniform distribution on $\oi$ as in \refE{E1u}; furthermore
\begin{equation}
X'-X
=\frac1{1-e^{-Y}}-\frac1Y
=\frac{e^{-Y}-1+Y}{(1-e^{-Y})Y}.
\end{equation}
This is a positive \rv{} with finite expectation
\begin{equation}
  \E(X'-X)
=
\intoo\frac{e^{-y}-1+y}{(1-e^{-y})y}e^{-y}\dd y
=
\intoo\Bigpar{\frac{e^{-y}}{1-e^{-y}}
-\frac{e^{-y}}{y}}\dd y
=\ggam,
\end{equation}
see \eg{} \cite[(5.9.18)]{DLMF}
or \cite{SJN11}.

Taking \iid{} pairs $(X_i,X'_i)\eqd(X,X')$ we thus have, with $S'_n\=\sumin
X'_i$, by the law of large numbers,
\begin{equation}
  \frac{S'_n-S_n}{n}\pto \E(X'-X)=\ggam.
\end{equation}
Since \refE{E1u} shows that $S'_n/n-\log n\dto \SSSx1(\pi/2,1,1-\ggam)$,
it follows that
\begin{equation}
  S_n/n-\log n\dto \SSSx1(\pi/2,1,1-2\ggam).
\end{equation}
We thus have \eqref{emm1} with
\begin{equation}
  b_n=\log n + 1 -2\ggam + o(1).
\end{equation}
\end{example}

\subsection{The case $\ga=2$}

If $\ga=2$, then $a_n$ in \eqref{da} have to be chosen such that
\begin{equation}\label{fornby}
  \frac{n\mu(a_n)}{a_n^2}\to C
\end{equation}
for some $C>0$, see \cite[(5.23)]{FellerII}; conversely any such sequence
$(a_n)$ will do.

\begin{theorem}\label{Tlim2}
  If $\mu(x)$ is slowly varying with $\mu(x)\to\infty$ as \xtoo{} and
  \eqref{fornby} holds, then
  \begin{equation}
	\frac{S_n-\E S_n}{a_n}\dto N(0,C).
  \end{equation}
\end{theorem}
\begin{proof}
  \citet[Theorem XVII.5.3]{FellerII}.
\end{proof}

\begin{example}\label{Epower2}
  Suppose that $\ga=2$ and that $X$ is a random variable such that, as \xtoo,
  \begin{equation}
	\P(X>x)\sim C x^{-2},
  \end{equation}
with $C>0$,
and $\P(X<-x)=o(x^{-2})$.
Then \eqref{td1b} holds with $p_+=1$, and thus
$p_-\=1-p_+=0$. 
Furthermore, as \xtoo,
\begin{equation}
  \begin{split}
\mu(x)&=\E\Bigpar{\int_0^{|X|}2t\dd t \,\ett{|X|\le x}}	
=
\E\int_0^{x}\ett{t\le |X|\le x}2t \dd t
\\&
=\int_0^{x}2t \P(t\le |X|\le x)\dd t
=\int_0^{x}2t \P(|X|>t)\dd t - x^2\P(|X|>x)
\hskip-5em
\\&
=\etto\int_1^{x}2t Ct\qww\dd t +O(1)
\sim 2C\log x.
  \end{split}
\end{equation}
Thus \eqref{td1a} holds with $L_1(x)=2C\log x$.

We take $a_n\=\sqrt{n\log n}$. Then $\mu(a_n)\sim 2C\tfrac12\log n=C\log n$,
so \eqref{fornby} holds and \refT{Tlim2} yields
  \begin{equation}
	\frac{S_n-\E S_n}{\sqrt{n\log n}}\dto N(0,C).
  \end{equation}
\end{example}

\section{Attraction and \chf{s}}

We study the relation between the attraction property \eqref{da} and the
\chf{} $\gf_X(t)$ of $X$. For simplicity, we consider only the common case when
$a_n=n\xga$. Moreover, for simplicity we state results for $\gf_X(t)$,
$t>0$ only, recalling \eqref{ch+-} and $\gf_X(0)=1$.

\begin{theorem}
  \label{Tjepp}
Let $0<\ga\le2$.
The following are equivalent.
\begin{romenumerate}[-10pt]
\item 
$\displaystyle \frac{S_n}{n\xga}\dto Z$ for some non-degenerate \rv{} $Z$.
\item 
The \chf{} $\gf_X$ of $X$ satisfies
\begin{equation}
  \label{tjeppii}
\gf_X(t)=1-(\kk-\ii\tau)t^\ga + o(t^\ga)
\qquad \text{as } t\downto0,
\end{equation}
for some real $\kk>0$ and $\tau$.
In this case, $Z$ is strictly $\ga$-stable 
and has the \chf{} \eqref{jepp1}.
(Hence, $|\tau|\le \kk\tan\piaq$.) 
\end{romenumerate}
\end{theorem}
\begin{proof}
  If (i) holds, then for every integer $m$,
  \begin{equation*}
\frac{S_{mn}}{(mn)\xga}
=\frac1{m\xga} \sum_{k=1}^m\frac{1}{n\xga}\sum_{j=1}^nX_{(k-1)n+j} 
\dto \frac1{m\xga} \sum_{k=1}^m Z_k,
\qquad \text{as \ntoo},
  \end{equation*}
with $Z_k\eqd Z$ \iid.
Since also $(mn)\xgaw S_{mn}\dto Z$, we have
$m\xgaw \sum_{k=1}^m Z_k\eqd Z$, and thus $Z$ is strictly $\ga$-stable.

We use \refC{CJ} and suppose that $Z$ has \chf{} \eqref{jepp1}. Then the
continuity theorem yields
\begin{equation}
  \label{jepp2}
\gf_X(t/n\xga)^n
\to \gf_Z(t)
=\exp\bigpar{-(\kk-\ii\tau)t^\ga},
\qquad t\ge0;
\end{equation}
moreover, this holds uniformly for, \eg, $0\le t\le 1$.

In some neighbourhood $(-t_0,t_0)$ of 0, $\gf_X\neq0$ and thus
$\gf_X(t)=e^{\psi(t)}$ for some continuous function $\psi:(-t_0,t_0)\to\bbC$
with $\psi(0)=0$.
Hence, \eqref{jepp2} yields (for $n>1/t_0$)
\begin{equation*}
  \exp\Bigpar{n\psi\Bigparfrac{t}{n\xga}+(\kk-\ii\tau)t^\ga}
=1+o(1),
\qquad \text{as \ntoo},
\end{equation*}
uniformly for $0\le t\le 1$, which implies
\begin{equation*}
n\psi\Bigparfrac{t}{n\xga}+(\kk-\ii\tau)t^\ga
=o(1),
\qquad \text{as \ntoo},
\end{equation*}
since the \lhs{} is continuous and 0 for $t=0$,
and thus 
\begin{equation}\label{jepp3}
\psi\parfrac{t}{n\xga}+(\kk-\ii\tau)\frac{t^\ga}{n}
=o(1/n),
\qquad \text{as \ntoo},
\end{equation}
uniformly for $0\le t\le 1$.

For $s>0$, define $n\=\floor{s^{-\ga}}$ and $t\=sn\xga\in(0,1]$. As $s\downto0$,
we have $n\to\infty$ and \eqref{jepp3} yields
\begin{equation}
\psi(s)=-(\kk-\ii\tau)s^\ga+o(1/n)
=-(\kk-\ii\tau)s^\ga+o(s^\ga).
\end{equation}
Consequently, as $s\downto0$,
\begin{equation}
\gf_X(s)=e^{\psi(s)}
=1-(\kk-\ii\tau)s^\ga+o(s^\ga),
\end{equation}
so \eqref{tjeppii} holds.

Conversely, if \eqref{tjeppii} holds, then, for $t>0$,
\begin{equation*}
  \E e^{\ii tS_n/n\xga}
=\gf_X\bigpar{t/n\xga}^n
=\Bigpar{1-(\kk-\ii\tau+o(1))\frac{t^\ga}n}^n
\to\exp\bigpar{-(\kk-\ii\tau)t^\ga},
\end{equation*}
as \ntoo, and thus by the continuity theorem $S_n/n\xga\dto Z$, where $Z$ has
the \chf{} \eqref{jepp1}.
\end{proof}

For $\ga=1$, it is not always possible to reduce to the case when $b_n=0$ in
\eqref{da} and the limit is strictly stable. The most common case is covered
by the following theorem.

\begin{theorem}
  \label{TJ1}
The following are equivalent, for any real $b$.
\begin{romenumerate}[-10pt]
\item 
$\displaystyle \frac{S_n}{n}-b\log n\dto Z$ for some non-degenerate \rv{} $Z$.
\item 
The \chf{} $\gf_X$ of $X$ satisfies
\begin{equation}
  \label{tj1}
\gf_X(t)=1-(\kk-\ii\tau)t - \ii bt\log t + o(t)
\qquad \text{as } t\downto0,
\end{equation}
for some real $\kk>0$ and $\tau$.
In this case, $Z$ is $1$-stable 
and has the \chf{} \eqref{jepp11}.
(Hence, $|b|\le 2\kk/\pi$.) 
\end{romenumerate}
\end{theorem}

\begin{proof}
(ii)$\implies$(i).
If \eqref{tj1} holds, for any $\kk\in\bbR$, then, as $t\downto0$,
\begin{equation}
\log \gf_X(t)={-(\kk-\ii\tau+o(1))t - \ii bt\log t}
\end{equation}
and thus, as \ntoo, for every fixed $t>0$,
\begin{equation*}
  \begin{split}
  \E e^{\ii t(S_n/n-b\log n)}
&
=\gf_X\bigpar{t/n}^n e^{-\ii bt\log n}
\\&
=\exp\Bigpar{n\Bigpar{-(\kk-\ii\tau+o(1))\frac{t}n - \ii b\frac tn\log\frac tn}
-\ii bt\log n}
\\&
\to
\exp\bigpar{-(\kk-\ii\tau){t} - \ii b t\log t}
  \end{split}
\end{equation*}
which shows (i), where $Z$ has the \chf{} \eqref{jepp11}.

Furthermore, for use below, note that \eqref{jepp11} implies
$|\gf_Z(t)|=e^{-\kk t}$ for $t>0$. Since $|\gf_Z(t)|\le1$, this shows that
$\kk\ge0$. Moreover, if $\kk=0$, then $|\gf_Z(t)|=1$ for $t>0$, and thus for
all $t$, which implies that $Z=c$ \as{} for some $c\in\bbR$, so $Z$ is
degenerate and $b=0$. Hence, \eqref{tj1} implies $\kk\ge0$, and $\kk=0$ is
possible only when $b=0$  and $S_n/n\pto\tau$.

(i)$\implies$(ii). 
  Let $\gam_1\=|b|\pi/2$ and $\gb_1\=-\sgn b$.
Let $Y$ and $Y_i$ be \iid, and independent of $(X_j)_1^\infty$ and $Z$, with
distribution $\SSSS1{\gam_1}{\gb_1}0$. (If $b=0$ we simply take $Y_i\=0$.)
Then $Y_i$ has, by \eqref{chf}, the \chf{}
\begin{equation}
  \label{emm}
\gf_Y(t)=\exp\bigpar{-\gam_1t+\ii bt\log t},
\qquad t>0.
\end{equation}
By \refT{Tchf}(ii), 
\begin{equation}
  \sumin Y_i\eqd nY-bn\log n.
\end{equation}
Define $\tX_i\=X_i+Y_i$. Then,
\begin{equation}
  \frac1n\sumin\tX_i 
=
  \frac1n\sumin X_i +   \frac1n\sumin Y_i 
\eqd
\frac{S_n}n + Y-b\log n \dto Z+Y.
\end{equation}
Thus, by \refT{Tjepp}, for some $\kk_2>0$ and $\tau_2$, 
\begin{equation}
  \gf_X(t)\gf_Y(t)=\E e^{\ii t \tX_i}
=1-(\kk_2-\ii\tau_2)t+o(t)
\qquad \text{as }t\downto0,
\end{equation}
and hence, using \eqref{emm},
\begin{equation}
  \gf_X(t)=\E e^{\ii t \tX_i}/\gf_Y(t)=
1-(\kk_2-\ii\tau_2-\gam_1)t-\ii bt\log t +o(t),
\end{equation}
which shows \eqref{tj1}, with $\kk=\kk_2-\gam_1\in\bbR$.

Finally, we have shown in the first part of the proof that \eqref{tj1}
implies $\kk>0$, because $Z$ is non-degenerate.
\end{proof}

We can use these theorems to show the following.

\begin{theorem}\label{TY}
  Let $0<\ga\le2$.
Suppose that $X$ is such that
\begin{equation}
  n\xgaw\sumin X_i\dto Z,
\end{equation}
where $Z$ is an $\ga$-stable random variable with \chf{} \eqref{jepp1}
and that $Y\ge0$ is a random variable with $\E Y^\ga<\infty$.
Let $(Y_i)_1^\infty$ be independent copies of $Y$ that are independent of
$(X_i)_1^\infty$.
Then 
\begin{equation}\label{xp1}
  n\xgaw\sumin X_iY_i\dto Z'\=\bigpar{\E Y^\ga}\xga Z,
\end{equation}
where the limit $Z'$ has the \chf
\begin{equation}
  \label{xp2}
\gf_{Z'}(t)
=
\exp\bigpar{-(\E Y^\ga \kk-\ii\E Y^\ga\tau)t^\ga},
\qquad t\ge0.
\end{equation}
If\/ $Z\sim\SSSS\ga\gam\gb0$ (where $\gb=0$ if $\ga=1$), then 
$Z'\sim\SSSS\ga{(\E Y^\ga)\xga\gam}\gb0$.
\end{theorem}

\begin{proof}
  By \refT{Tjepp}, for $t\ge0$,
  \begin{equation}\label{np1}
	\gf_X(t)=1-(\kk-\ii\tau)t^\ga+t^\ga r(t),
  \end{equation}
where $r(t)\to0$ as $t\downto0$. Furthermore, \eqref{np1} implies that
$r(t)=O(1)$ as $t\to\infty$, and thus $r(t)=O(1)$ for $t\ge0$. 

Consequently, for $t>0$, assuming as we may that $Y$ is independent of $X$,
\begin{equation}
  \begin{split}
\gf_{XY}(t)
&=
\E e^{\ii tXY}
=
\E \gf_X(tY)
=	
\E\bigpar{1-(\kk-\ii\tau)t^\ga Y^\ga+t^\ga Y^\ga r(tY)}
\hskip-3em
\\
&=1-(\kk-\ii\tau)t^\ga \E Y^\ga+t^\ga \E\bigpar{Y^\ga r(tY)},
  \end{split}
\end{equation}
where $\E\bigpar{Y^\ga r(tY)}\to0$ as $t\downto0$ by dominated convergence;
hence 
\begin{equation}
  \begin{split}
\gf_{XY}(t)
&=1-(\kk-\ii\tau)t^\ga \E Y^\ga+o(t^\ga)
\qquad\text{as }t\downto0.  
\end{split}
\end{equation}
\refT{Tjepp} applies and shows that $n\xgaw\sumin X_iY_i\dto Z'$, where $Z'$
has the \chf{} \eqref{xp2}. Moreover, by \eqref{jepp1}, $(\E Y^\ga)\xga$ has
this \chf, so we 
may take $Z'\=(\E Y^\ga)\xga$.

The final claim follows by \refR{Rlinx}.
\end{proof}

\begin{theorem}\label{TY1}
Suppose that $X$ is such that, for some real $b$,
\begin{equation}
  n\qw\sumin X_i-b\log n\dto Z,
\end{equation}
where $Z$
is a $1$-stable random variable,
and that $Y\ge0$ is a random variable with $\E Y \log Y<\infty$.
Let $(Y_i)_1^\infty$ be independent copies of $Y$ that are independent of
$(X_i)_1^\infty$.
Then, with $\mu\=\E Y$, 
\begin{equation}\label{yp1}
  n\qw\sumin X_iY_i -b\mu\log n\dto 
Z'\=\mu Z-b\bigpar{\E(Y\log Y)-\mu\log \mu}.
\end{equation}

$Z$ has the \chf{} \eqref{jepp11} for some $\kk$ and $\tau$, and then
the limit $Z'$ has the \chf,
with $\nu\=\E(Y\log Y)$,
\begin{equation}
  \label{yp2}
\gf_{Z'}(t)
=
\exp\bigpar{-\bigpar{\mu\kk+\ii(b\nu-\mu\tau)t}-\ii b\mu t{\log t}},
\qquad t>0.
\end{equation}
If
$Z\sim\SSSS1\gam\gb\gd$, 
then 
$Z'\sim\SSSS1{\mu\gam}\gb{\mu\gd-b\nu}$. 
\end{theorem}

\begin{proof}
  By \refT{TJ1}, for $t\ge0$,
  \begin{equation}\label{up1}
	\gf_X(t)=1-(\kk-\ii\tau)t-\ii b t\log t+t r(t),
  \end{equation}
where $r(t)\to0$ as $t\downto0$; moreover
$Z$ has the \chf{} \eqref{jepp11}.
Furthermore, \eqref{up1} implies that
$r(t)=O(\log t)$ as $t\to\infty$, and thus $r(t)=O(1+\log_+t)$ for $t\ge0$. 

Consequently, for $t>0$, assuming as we may that $Y$ is independent of $X$,
\begin{equation*}
  \begin{split}
\gf_{XY}(t)
&=
\E \gf_X(tY)
\\
&=1-(\kk-\ii\tau)t \E Y-\ii b t \E\bigpar{Y\log(tY)}+t\E\bigpar{Y r(tY)},
\\
&=1-\bigpar{\mu\kk-\ii\mu\tau+\ii b\E(Y\log Y)}t-\ii b\mu t{\log t}
 +t\E\bigpar{Y r(tY)},
  \end{split}
\end{equation*}
where $\E\bigpar{Y r(tY)}\to0$ as $t\downto0$ by dominated convergence;
hence 
\begin{equation}
  \begin{split}
\gf_{XY}(t)
&=1-\bigpar{\mu\kk-\ii\mu\tau+\ii b\nu}t-\ii b\mu t{\log t}+o(t)
\qquad\text{as }t\downto0.  
\end{split}
\end{equation}
\refT{TJ1} applies and shows that $n\qw\sumin X_iY_i-b\mu\log n\dto Z'$,
where $Z'$ 
has the \chf{} \eqref{yp2}. 
Moreover, it follows easily from \eqref{jepp11} that
$\mu Z-b\bigpar{\E(Y\log Y)-\mu\log \mu}$ has this \chf, and thus \eqref{yp1}
follows. 

Finally, if $Z\sim\SSSS1\gam\gb\gd$, then $b=\frac2\pi\gb\gam$ by \refR{R11}
and it follows easily 
from \refR{Rlinx} that $Z'\sim\SSSS1{\mu\gam}\gb{\mu\gd-b\nu}$;
alternatively, it follows
directly from \eqref{yp1} and
\eqref{chf} that $Z'$ has the \chf
\begin{equation}
  \begin{split}
  \gf_{Z'}(t)
&=\gf_Z(\mu t)\exp\bigpar{-\ii b t(\nu-\mu\log \mu)}
\\&
=
\exp\Bigpar{-\gam\mu|t|
\Bigpar{1+\ii\gb\frac2{\pi}\sgn(t)\log|t|}
+\ii\gd \mu t-\ii b t \nu}.	
  \end{split}
\end{equation}
\end{proof}

\begin{example}
  Let $X\=U/U'$, where $U,U'\sim \U(0,1)$ are independent.
By \refE{E1u} and \refT{TY1}, with 
$Z\sim\SSSS1{\pi/2}1{1-\ggam}$, $b=1$,
$\mu\=\E U=1/2$
and
\begin{equation}
\nu\=\E U\log U
=\intoi x\log x\dd x
=-\frac14,
\end{equation}
we obtain
\begin{equation}
  \frac{S_n}n-\frac12\log n\dto \frac12Z - \nu +\frac12\log\frac12
=\frac12Z+\frac14-\frac12\log2
\sim\SSSx1\Bigpar{{\frac{\pi}4},1,{\frac34-\frac{\ggam}2}}.
\end{equation}
\end{example}

\begin{example}
  Let $X\=Y/Y'$ where $Y,Y'\sim\Exp(1)$ are independent.
(Thus $X$ has the $F$-distribution $F_{2,2}$.)
By \refE{E1exp} and \refT{TY1}, with 
$Z\sim\SSSS1{\pi/2}1{1-2\ggam}$,
$b=1$,
$\mu\=\E Y=1$ and
\begin{equation}
\nu\=\E Y\log Y
=\intoo x\log x \,e^{-x}\dd x
=\Gamma'(2)=1-\ggam  ,
\end{equation}
we obtain
\begin{equation}
  \frac{S_n}n-\log n\dto Z - \nu 
=Z-1+\ggam
\sim\SSSS1{\pi/2}1{-\ggam}.
\end{equation}
This is in accordance with \refE{E1u}, since, as is well-known,
$U\=Y'/(Y+Y')\sim \U(0,1)$, and thus we can write $X=(Y+Y')/Y'-1=1/U-1$.
\end{example}

\begin{example}
  Let $X\=V^2/W$ where $V\sim \U(-\frac12,\frac12)$ and $W\sim\Exp(1)$ are
  independent. 
By \refE{E1exp} and \refT{TY1}, with 
$Z\sim\SSSS1{\pi/2}1{1-2\ggam}$,
$b=1$,
$\mu\=\E V^2=1/12$ and
\begin{equation}
  \begin{split}
\nu&\=2\E V^2\log |V|
=4\int_0^{1/2} x^2\log x\dd x
=4\lrsqpar{\frac{x^3}3\log x-\frac{x^3}9}_0^{1/2}
\\&\phantom:
=-\frac{3\log 2+1}{18},	
  \end{split}
\end{equation}
we obtain
\begin{equation}
  \frac{S_n}n-\frac1{12}\log n\dto
\frac1{12}Z-\nu+\frac1{12}\log\frac1{12}
\sim\SSSx1\Bigpar{{\frac{\pi}{24}},1,{\frac{5-6\ggam+6\log2}{36}}}.
\end{equation}
Equivalently, using \refR{Rlinx},
\begin{equation}
  \frac{24 S_n}{\pi n}-\frac2{\pi}\log n\dto
\frac2{\pi}Z-\frac{24\nu}\pi-\frac2{\pi}\log12
\sim\SSSx1\Bigpar{{1},1,{\frac2\pi\Bigpar{\frac53-2\ggam+\log\frac\pi6}}}.
\end{equation}
This is shown directly in
\citet[Theorem 5.2 and its proof]{HPS2004}.
\end{example}

\begin{example}
More generally, let $X\=V^2/W$ where $V\sim \U(q-1,1)$ and $W\sim\Exp(1)$ are
  independent, for some fixed real $q$.
By \refE{E1exp} and \refT{TY1}, with 
$Z\sim\SSSS1{\pi/2}1{1-2\ggam}$,
$b=1$,
\begin{equation}
  \mu=\E V^2
=\Bigpar{\E V}^2+\Var V
=\Bigpar{q-\frac12}^2+\frac1{12}
=\frac{3q^2-3q+1}{3}
\end{equation}
and
\begin{equation}
  \begin{split}
\nu&\=2\E V^2\log |V|
=2\int_{q-1}^{q} x^2\log |x|\dd x
=2\lrsqpar{\frac{x^3}3\log |x|-\frac{x^3}9}_{q-1}^{q}
\\&\phantom:
=2\frac{q^3\log |q|+(1-q)^3\log|1-q|}{3}-2\frac{3q^2-3q+1}{9}  ,
\end{split}
\end{equation}
we obtain
\begin{equation}
  \frac{S_n}n-\mu\log n\dto
\mu Z-\nu+\mu\log\mu
\sim\SSSx1\Bigpar{{\mu\frac{\pi}{2}},1,{(1-2\ggam)\mu-\nu}}.
\end{equation}
Equivalently, using \refR{Rlinx},
\begin{equation}
  \frac{S_n-n (\E V)^2}{\mu n}-\log n\dto
Z-\frac{\nu}{\mu}+\log\mu-\frac{(\E V)^2}{\mu}
\sim\SSSx1\Bigpar{{\frac{\pi}{2}},1,b_q},
\end{equation}
with
\begin{equation}
  b_q\=
\frac23-2\ggam
-2\frac{q^3\log |q|+(1-q)^3\log|1-q|}{3q^2-3q+1}+\log\frac{3q^2-3q+1}{3}
+\frac1{12\mu}  .
\end{equation}

This is shown (in the case $0\le q\le 1$)
directly in
\citet[Theorem 4.2 and its proof]{HPS2005}.
\end{example}

\newcommand\AAP{\emph{Adv. Appl. Probab.} }
\newcommand\JAP{\emph{J. Appl. Probab.} }
\newcommand\JAMS{\emph{J. \AMS} }
\newcommand\MAMS{\emph{Memoirs \AMS} }
\newcommand\PAMS{\emph{Proc. \AMS} }
\newcommand\TAMS{\emph{Trans. \AMS} }
\newcommand\AnnMS{\emph{Ann. Math. Statist.} }
\newcommand\AnnPr{\emph{Ann. Probab.} }
\newcommand\CPC{\emph{Combin. Probab. Comput.} }
\newcommand\JMAA{\emph{J. Math. Anal. Appl.} }
\newcommand\RSA{\emph{Random Struct. Alg.} }
\newcommand\ZW{\emph{Z. Wahrsch. Verw. Gebiete} }
\newcommand\DMTCS{\jour{Discr. Math. Theor. Comput. Sci.} }

\newcommand\AMS{Amer. Math. Soc.}
\newcommand\Springer{Springer-Verlag}
\newcommand\Wiley{Wiley}

\newcommand\vol{\textbf}
\newcommand\jour{\emph}
\newcommand\book{\emph}
\newcommand\inbook{\emph}
\def\no#1#2,{\unskip#2, no. #1,} 
\newcommand\toappear{\unskip, to appear}

\newcommand\urlsj[1]{\url{http://www.math.uu.se/~svante/papers/#1}}
\newcommand\arxiv[1]{\url{arXiv:#1.}}
\newcommand\arXiv{\arxiv}

\def\nobibitem#1\par{}

\end{document}